\documentclass[10pt]{amsart}

\usepackage{amsmath}
\usepackage{hyperref} 

\usepackage[style=alphabetic,
, maxbibnames=99]{biblatex}

\usepackage{amssymb}
\usepackage{mathtools}      % improves amsmath
\usepackage{mathabx} % makes many symbols (as ≤\leq) more beautiful; must be loaded after mathtools 
\usepackage{enumitem}
\usepackage{amsthm}
\usepackage{thmtools}
\usepackage{esint} 
\usepackage[font=small]{caption}
\usepackage{xcolor}
\declaretheorem[numberwithin=section]{theorem} 
\declaretheorem[sibling=theorem]{proposition} 
\declaretheorem[sibling=theorem]{corollary} 
\declaretheorem[sibling=theorem]{lemma}

\declaretheorem[sibling=theorem]{conjecture}

\declaretheorem[sibling=theorem, style=remark]{claim}

\declaretheorem[sibling=theorem, style=definition]{definition}
\declaretheorem[sibling=theorem, style=remark]{remark}

\numberwithin{equation}{section}     % rabic*)},ref=\arabic*}
\setlist[enumerate,1]{label={\upshape \arabic*.},ref=\alph*}

\newcommand{\N}{\mathbb{N}} \newcommand{\Z}{\mathbb{Z}} \newcommand{\Q}{\mathbb{Q}} \newcommand{\R}{\mathbb{R}}
\newcommand{\C}{\mathbb{C}}

\newcommand{\T}{\mathbb{T}}
\renewcommand{\H}{\mathbb{H}}
\newcommand{\V}{\mathcal{V}}
\newcommand{\cB}{\mathcal{B}}\newcommand{\cC}{\mathcal{C}}
\newcommand{\cD}{\mathcal{D}}\newcommand{\cE}{\mathcal{E}}\newcommand{\cF}{\mathcal{F}}

\newcommand{\cP}{\mathcal{P}}

\newcommand{\cV}{\mathcal{V}}

\newcommand{\Haar}{\text{Haar}}
\newcommand{\D}{\mathbb{D}}

\renewcommand{\P}{\mathbb{P}}
\newcommand{\CP}{\mathbb{CP}}
\newcommand{\Sym}{\mathrm{Sym}}
\newcommand{\F}{\mathcal{F}}
\newcommand{\FF}{\mathbb{F}}
\newcommand{\RP}{\mathbb{RP}}

\newcommand{\SL}{\mathrm{SL}}
\newcommand{\SO}{\mathrm{SO}}

\newcommand{\GL}{\mathrm{GL}}
\newcommand{\SU}{\mathrm{SU}}
\newcommand{\U}{\mathrm{U}}
\newcommand{\Grass}{\mathrm{Grass}}
\newcommand{\Diff}{\mathrm{Diff}}
\newcommand{\Prob}{\mathrm{Prob}}

\DeclareMathOperator{\trace}{tr}

\newcommand{\tribar}[1]{\mathopen{| {\kern -1.5pt} | {\kern -1.5pt} |} {#1} \mathclose{| {\kern -1.5pt} | {\kern -1.5pt} |}}

\ifluatex 
	  
\else
	  
\fi

\usepackage{bbm}

\renewcommand{\setminus}{\smallsetminus}

\renewcommand{\epsilon}{\varepsilon}
\usepackage{tikz-cd}

\title{Dedieu--Shub Measures}
\author{joshua Paik}
 \email{jdp5887@psu.edu}
 \address{The Pennsylvania State University}
\date{\today}

\begin{document}

\begin{abstract}
This paper introduces Dedieu-Shub measures and surveys their appearance in the literature. 
\end{abstract}

\maketitle

% \tableofcontents

\section{Introduction}

\subsection{Setting and Results}

 Let $G$ be a locally compact second countable topological group\footnote{more generally, we can consider semigroups}. Let $K$ be a compact subgroup of $G$. Suppose both $G$ and $K$ act transitively on a compact metric space $X$ equipped with a Borel sigma algebra. Then $X$ is a homogeneous space and $X = G/G' = K/K'$. In this paper, we consider specific cases when $G = \GL(d,\C), \ \GL(2,\R)$ and $\Diff^\infty(\R/\Z)$, $K = \SU(d), \ \SO(2),$ and $\R/\Z$ and $X = \FF_d, \ \RP^1,$ and $\R/\Z$ respectively. We denote normalized Haar measure on $K$ by $\nu$ and define $\mu$ to be the unique normalized $K$--stationary measure on $X$ -- meaning $\mu$ satisfies $\mu(B) = \int \mu(kB) \ d\nu(k)$ for all $B \in \cB(X)$ and $\mu(X) = 1$. Let $\Prob(X)$ denote the space of all Borel probability measures on $X$ endowed with the weak-$*$ topology. The purpose of this paper is to explore the following definition.

\begin{definition} \label{definition}
Let $G$ act on $X$ and suppose $K \subset G$ is a compact subgroup equipped with Haar measure $\nu$. We call a Borel measurable function $m: G \to \Prob(X)$ a \textit{Dedieu--Shub measure} for $(K,G,X)$ if 
\begin{enumerate}
\item For every $g \in G$, we have that $m_g$ is a $g$--invariant probability measure on $X$ and
\item \textit{(Dedieu--Shub property)}for every $g \in G$ and for every $B \in \cB(X)$, we have $$\int_K m_{kg}(B) \ d\nu(k) = \mu(B).$$
\end{enumerate}
\end{definition}

This definition is due to Jairo Bochi. Naming such a function after Dedieu and Shub follows from a key technical result in \cite{dedieu2003mike}.\footnote{We remark that the definition did not appear in the work of Shub and his collaborators. } As with all definitions, one is concerned that it is well defined, that there are non trivial examples of things satisfying the definition, and applications. This is what this paper is concerned with. 

Our first example of a Dedieu--Shub measure is over $\GL(d,\C)$ and is due to Dedieu and Shub. Let $\FF_d$ denote the complex flag variety isomorphic to $\GL(d,\C)/P = \SU(d)/\T^d$ where $P$ is the subgroup of upper triangular matrices and $\T^d$ %$ 
is the maximal torus of $\SU(d)$. A more explicit description of the complex flag variety is  
\begin{align*}
    \FF_d &= \{ \ [\langle w_1\rangle \subset \langle w_1,w_2\rangle  \subset ... \subset \langle w_1,...,w_d\rangle ] : \\ & \ \ \ \ \ \ \ \ \ \ \ \ \ \ \ \ [w_1,...,w_d] \text{ is a linearly independent set } \}\\
    &\subseteq \text{Grass}(1,d) \times \text{Grass}(2,d) \times ... \times \text{Grass}(d,d)
\end{align*}
where $\langle \cdot \rangle$ 
denotes the span of $\cdot.$

 It is useful notationally to define a function $\phi: \{ \text{linearly independent set}\} \to \FF_d$ where $$\phi(w_1,...,w_d) = [\langle w_1 \rangle \subset \langle w_1, w_2 \rangle  \subset ... \subset \langle w_1, ..., w_d \rangle ].$$

\begin{theorem}[\cite{dedieu2003mike}]
\label{DS measures}
Let $G = GL(d,\C)$, $K = \SU(d)$, and $X = \FF_d$.  Let $\Sym(d)$ be the symmetric group on $\{1,...,d\}$. For a given $A \in \GL(d,\C)$, let 
\begin{align*}
\lambda_i(A) &= \text{ the ith largest in modulus eigenvalue of } A\\
v_i(A) &= \text{the eigenvector corresponding to } \lambda_i(A).
\end{align*}
Let $p:\Sym(d) \times GL(d,\C) \to [0,1]$ be defined as $$p_\sigma(A) = \frac{\prod \limits_{j = 1}^d |\lambda_j|^{2(d - \sigma(j))}}{\sum \limits_{\pi \in \Sym(d)} \prod \limits_{j=1}^d |\lambda_j|^{2(d - \pi(j))}}.$$
Define $m: \GL(d,\C) \to \Prob(\FF_d)$ as $$m_A = \sum \limits_{\sigma \in \Sym(d)} p_\sigma(A) \delta_{\phi(v_{\sigma(1)}(A),...,v_{\sigma(d)}(A))}.$$
Then $m_A$ is a Dedieu--Shub measure. 
\end{theorem} 

For example, the Dedieu--Shub measure for $\GL(2,\C)$ is \begin{equation}
    m_A = \frac{|\lambda_1|^2}{|\lambda_1|^2 + |\lambda_2|^2} \delta_{v_1} + \frac{|\lambda_2|^2}{|\lambda_1|^2 + |\lambda_2|^2} \delta_{v_2}. \label{dim2weights}
\end{equation}  

This gives the following result of Dedieu--Shub. 

\begin{theorem}[Dedieu--Shub \cite{dedieu2003mike}] \label{originalineq}   
Let $A \in GL(d,\C)$. Let $f:\R \to \R$ be monotone increasing. Then 
 $$\int_{U(d)} f \left(\prod \limits_{i=1}^k |\lambda_i (UA )| \right) \ d\nu(U) \geq \int_{\text{Grass}(d,k)} f \left( \det A|g_{d,k} \right) \ d\mu(g_{d,k})$$
\end{theorem}
In particular, the theorem is true when $f$ is $\log$, which is interesting dynamically. 

One can naturally ask whether similar results hold over $GL(d,\R)$. However, constructing Dedieu--Shub measures over $\GL(d,\R)$ is more difficult than over $\GL(d,\C)$. One reason is that the invariant measures for two matrices in $SO(d)A$, where $A \in \GL(d,\R)$, do not necessarily have disjoint supports. Even in dimension 2, a family $SO(2)A$ acting on $\RP^1$ has both \textit{elliptic matrices} whose invariant measures are (typically) fully supported and \textit{hyperbolic matrices} whose invariant measures are supported on points. In contrast, the invariant measures of two different elements from $U(d)A$ acting on $\FF_d$ are typically Dirac masses with disjoint supports.

Our second example, is the case of the linear projective action on $\RP^1$. 
    
\begin{theorem}
\label{GL2R}
    Let $G = \text{GL}(2,\R)$, $K = SO(2), X = \RP^1$. Define a function $m:G \to Prob(X)$ as follows:
    \begin{enumerate}
        \item[\textit{(1)}] (hyperbolic case) if $\det A > 0$ and the eigenvalues of $A$ satisfy $|\lambda_1| > |\lambda_2|$, then $$m_A := \delta_{v_1}.$$
        \item[\textit{(2)}] (elliptic case) if $\det A > 0$ and the eigenvalues are \textit{not} real, then $$m_A := \text{the unique acip measure conjugate to an irrational rotation}.$$
        \item[\textit{(3)}] if $\det A <0$, then necessarily the eigenvalues of $A$ are real and $$m_A = \frac{|\lambda_1|}{|\lambda_1| + |\lambda_2|} \delta_{v_1} + \frac{|\lambda_2|}{|\lambda_1| + |\lambda_2|} \delta_{v_2}.$$
    \end{enumerate}

    Then $m_A$ is a Dedieu--Shub measure. 
\end{theorem}
The case of positive determinant is implicit in \cite{pujals2006expanding} and we present the argument here to be self contained. The case of negative determinant is potentially new. Producing Dedieu--Shub measures over $\GL(d,\R)$ for $d\geq 3$ seems like a challenging problem.

One reason Dedieu--Shub measures are interesting, is they are a tool in proving inequalities relating \textit{random} and \textit{deterministic} exponents. See section \ref{definititiondef} for more. Over $GL(d,\R)$, this is a question that apperars (as a comment) in Dedieu--Shub \cite{dedieu2003mike} and a survey article (as a question) of Burns, Pugh, Shub, Wilkinson \cite{burns2001recent}. Recall that a uniform random $k$--dimensional Grassmanian in $\C^d$, denoted $g_{k}$, can be represented as the span of the first $k$--columns of a Haar random $U \in U(d)$ or $O(d)$, and we denote this random $d \times k$ rectangular matrix as $U_k$. Then for a square matrix $A$, define $\det A|g_k := \det (AU_k)(AU_k)^*$. 

\begin{conjecture}[\cite{dedieu2003mike},\cite{burns2001recent}] \label{conjecture}
     Let $A \in GL(d,\R) \setminus \R I$. Then
     \begin{enumerate}
         \item $$\int_{\SO(d)} \prod \limits_{i=1}^k |\lambda_i (UA )| \ d\nu(U) \geq c_{d,k} \int_{\text{Grass}(d,k)} \det A|g_{k} \ d\mu(g_{k}),$$
         \item $$\int_{\SO(d)} \sum \limits_{i=1}^k \log |\lambda_i (UA )| \ d\nu(U) \geq c_{d,k} \int_{\text{Grass}(d,k)} \log \det A|g_{k} \ d\mu(g_{k})\text{ and,}$$
         \item One can choose $c_{d,k} = 1$ for all $d$ and $k$. 
     \end{enumerate}
\end{conjecture}

All experimental evidence suggests that $c_{d,k} = 1$. Rather recently, Armentano, Chinta, Sahi and Shub \cite{armentanorandom} recently proved the following result towards proving Conjecture \ref{conjecture}. 

\begin{theorem}[
 \cite{armentanorandom}]
Let $G = \SL(d,\R)$,$K = SO(d)$, and $X = \mathcal{G}(d,k)$, the space of real $k$--dimensional Grassmanians of $\R^d$. Then for every $k \in [1,..,d]$ we have that 

$$\int_{SO(d)} \sum_{i=1}^k \log |\lambda_i(OA)| d\nu(O) \geq c_{d,k} \int \log \det A | g_{k} \ d\mu(g_{k})$$

 where $c_{d,k} = \frac{1}{\binom{d}{k}}$.
\end{theorem}

When $k=1$, Rivin \cite{rivin2005CMP} also proved an inequality with a weaker constant. In particular, for both \cite{armentanorandom} and \cite{rivin2005CMP}, for every $k$, as $d$ increases, their $c_{d,k} \to 0$.

We conclude our paper by presenting experimental evidence that a Dedieu--Shub measures \textit{cannot exist} for $\text{Diff}^\infty(S^1)$ by studying a restricted example -- the Arnold family. The Arnold family $$f_{c,\varepsilon}(x) = x + c + \varepsilon \sin 2 \pi x \mod 1$$ when $(c,\varepsilon) \in [0,1] \times (0,1/2\pi].$ This is somewhat surprising in light of positive results for Blaschke products by Pujals, Roberts, and Shub \cite{pujals2006expanding}, however, also not surprising, in light of \cite{dLL2008entropy} and \cite{ledrappier2003random}.

% \begin{lemma}
% Let $Y$ and $X$ be compact metric spaces which come with uniform distributions $\mu_Y$ and $\mu_X$ respectively. Let $m: Y \to \text{Prob}(X)$ be measurable. For all $B \in \cB(X)$, the function $$g \in Y \to m(g)(B) \in \R$$ is Borel measurable.
% \end{lemma}

% \begin{proof}
%     Fix a $B \in \cB(X)$. We know both $m(g) \to \text{Prob}(X)$ and $\mu \in \text{Prob}(X) \to \mu(B) \in \R$ are measurable, and the map in question is simply a composition of these two. 
% \end{proof}

\section{General Properties of Dedieu--Shub measures and Application to Lyapunov Exponents}

We detail some properties of Dedieu--Shub measures and we use these to prove (known) Dedieu--Shub inequalities. 

\subsection{General properties}
\subsubsection{Dedieu--Shub measures are well defined}

\begin{lemma}
Let $Y$ and $X$ be compact metric spaces which come with uniform distributions $\mu_Y$ and $\mu_X$ respectively. Let $m: Y \to \text{Prob}(X)$ be Borel measurable. For all $B \in \cB(X)$, the function $$g \in Y \mapsto m(g)(B) \in \R$$ is Borel measurable.
\end{lemma}

\begin{proof}
    This is Lemma 2.2 in \cite{avila2012nonuniform}.
\end{proof}
\noindent

\subsubsection{Fubini--Like property} 
For a $B \in \cB(X)$, let $f = \mathbbm{1}_B$. Then Property \ref{definition}.2 becomes $$\int_K \int_X f(x) \ dm_{kg}(x) d\nu(k) = \int_X f \ d\mu.$$ As this holds for all measurable $B \in \cB(X)$, we have the following. 

\begin{lemma} \label{fubini}
    For every bounded and measurable $f:X \to \R$,  $$\int_K \int_X f(x) \ dm_{kg}(x) d\nu(k) = \int_X f \ d\mu.$$
\end{lemma}

\begin{proof}
    Our comment above implies the result when $f$ is a simple functions. Now, a bounded measurable function can be approximated by an increasing sequence of simple functions $\{f_n\}$ such that $f_n \to f$ pointwise and $|f_n(x)| \leq \|f\|_\infty$ for all $x$. Now apply dominated convergence theorem.
    
    % $\|f_n - f\|_1 \to 0$. The triangle inequality for $\|\cdot\|_1$ gives that $$\|f_n\|_1 \leq \|f_n\|_1 + \|f_n - f\|_1 \implies\ \left| \|f_n\|_1 - \|f\|_1 \right| \leq  \underbrace{\|f_n - f\|_1}_{\to 0} .$$
    % So $ \left| \|f_n\|_1 - \|f\|_1 \right| \to 0$.
\end{proof}

\subsubsection{Projecting Dedieu--Shub measures}

\begin{lemma} \label{projection}
    Let $\Pi:X \to Y$ be a continuous onto map. Dedine a $G$--action on $Y$ such that $g(\Pi(x)) := \Pi(g(x))$. Suppose $G$ acts on X preserving fibers of $\Pi$, that is $$\text{if } \Pi x_1 = \Pi x_2 \implies \Pi (gx_1) = \Pi (gx_2)$$ for all $x_1,x_2 \in X$  and $g \in G$. 
    Then, if $m$ is a Dedieu -- Shub measure for $(K,G,X)$, then $\tilde{m}_g := \Pi_*(m_g)$ is a Dedieu--Shub measure for $(K,G,Y)$. 
\end{lemma}

\begin{proof}$ $

    \begin{enumerate}
        \item[1.] \textit{(Invariance)} We wish to show for all $A \in \cB(Y)$, we have $\tilde{m}_g(g^{-1}A) = \tilde{m}_g(A)$. We have \begin{align*}
            \tilde{m}_g(g^{-1}A) = m_g(\Pi^{-1}(g^{-1}A)) = m_g((\Pi^{-1} \circ g^{-1})(A)) = m_g(g^{-1}(\Pi^{-1}(A))) \\
            = m_g(g^{-1}(\Pi^{-1}(A))) = m_g(\Pi^{-1}(A)) = \tilde{m}_g(A). 
        \end{align*}
        \item[2.] \textit{(Average)} We need to argue that for every $g \in G$, we have $\int_{k \in K} \tilde{m}_{kg} \ d\nu(k) = \Pi_*(\mu)$. Suppose not -- then for some Borel $B \in \cB(Y)$, $$\int_{k \in K} \tilde{m}_{kg}(B) \ d\nu(k) \neq \Pi_*(\mu)(B).$$ However, this would imply that for $$\int_{k \in K} m_{kg}(\Pi^{-1}B) \ d\nu(k) \neq \mu(\Pi^{-1} B),$$ contradicting that $m$ was originally Dedieu--Shub. 
    \end{enumerate}

\end{proof}

For example, projecting the Dedieu--Shub measure of Theorem \ref{DS measures} for $G = \GL(d,\C), K = \SU(d), X = \FF_d$ to  $G = \GL(d,\C), K = \SU(d), X =\CP^{d-1}$, we have the following.

\begin{corollary}[Projecting Theorem \ref{DS measures} to $\CP^{d-1}$] \label{CPmeasure}
    Fix $A \in \GL(d,\C)$ . Let $\nu$ be the Haar measure on $SU(d)$ and let $m$ be Haar measure on $S^{2d-1}$. There exists a function $p:GL(d,\C) \to \Delta^{d-1}$ the space of probability vectors of length $d$, so that for any Borel set $B \in \cB(\CP^{d-1})$, $$\int_{U(n)}  p_1(UA)\delta_{v_1(UA)} (B) + ... + p_d(UA)\delta_{v_d(UA)}(B) dU = \mu(B),$$ where $v_i(\cdot) = \text{the eigendirection in } \CP^{d-1} \text{ corresponding to the ith largest eigenvalue of }  \cdot$ and the value of $p_i$ is given by 

    $$p_i = \frac{\sum \limits_{\{\sigma \in \text{Sym}(d): \sigma(i) =  1\}} \prod \limits_{j = 1}^d |\lambda_j|^{2(d - \sigma(j))}}{\sum \limits_{\sigma \in \text{Sym}(d)} \prod \limits_{j=1}^d  |\lambda_j|^{2(d-\sigma(j))}}.$$
\end{corollary}

\subsection{An application to Lyapunov exponents} \label{definititiondef}

\subsubsection{} The \textit{random exponent} of a compactly supported measure $\gamma$ supported on $\Diff^\infty(M)$ is $$\text{RE}(\gamma) = \int_{\mathrm{supp}(\gamma)^{\N} \times X} \lim \limits_{n\to\infty} \frac{1}{n}\log \left \| \prod \limits_{k=1}^n Df_i (x) \right \| \ d\gamma(f_i) d\mu(x),$$ where $\mu$ is the volume. 

\subsubsection{} The \textit{mean exponent}  of a compactly supported measure $\gamma$ supported on $\Diff^\infty(M)$ is $$\text{LE}(\gamma) = \int_{\mathrm{supp}(\gamma) \times X} \lim \limits_{n \to \infty} \frac{1}{n} \log \|Df^n(x)\| \ d\gamma(f) \times d \mu(x).$$

\subsubsection{}

The original goal of Shub in a series of works \cite{dedieu2003mike}, \cite{ledrappier2003random}, \cite{pujals2006expanding}, \cite{pujals2008dynamics}, \cite{dLL2008entropy}, \cite{armentanorandom} as explained in his ICM paper \cite{shub2006all} and survey paper \cite{burns2001recent}, is to ascertain when -- given a measure $\gamma$ as above -- we have \begin{equation}
    \text{LE}(\gamma) \geq \text{RE}(\gamma) > 0. \label{inequality}
\end{equation} Characterizing for which measures such inequalities hold is a very interesting question. A candidate for $\gamma$ as proposed by Shub is a measure uniformly supported on a coset $Kg$ when $K$ is the isometry group of $M$ and $g$ has entropy. In this situation it is often easy to prove that $\text{RE}(\gamma) > 0$. To illustrate this point, we have the following. 

\begin{proposition}[\cite{rivin2004distortion}]
    Let $A \in \SL(d,\R)$. Let $\gamma$ be the pushforward of Haar measure $\nu$ onto $SO(d)A$. Then \begin{align*}
       \mathrm{RE}(\gamma) &= \int_{\mathrm{supp}(\gamma)^\N \times S^{n-1} } \lim \limits_{n \to \infty} \frac{1}{n} \log \left\| \prod \limits_{i=1}^n O_i A \right\| \ d\nu(\theta_i) \ d\mu(v)\\
       &=
       \int_{S^{n-1}} \log \|Av\| \ d\mu(v)  \geq 0 
    \end{align*}  with equality if and only if all singular values of $A$ are all $1$. 
\end{proposition}

\begin{proof}
    The first equality follows by the chain rule and that $SO(d)$ acts by isometries. It is sufficient to prove this for $A$ diagonal matrix of singular values. So let $A = \mathrm{diag}(a_1,...,a_d)$. Now using Jensen's inequality gives \begin{align*}
        \int \log \|Av\| \ d\mu(v) &= \int_{S^{d-1}} \frac{1}{2} \log (a_1^2v_1^2 + ... + a_d v_d^2) \ d\mu(v)\\
        & \geq \frac{1}{2} \int_{S^{d-1}}\log(a_1^2)v_1^2 + ... + \log(a_d^2)v_d^2 \ d\mu(v) \\
        &= \frac{1}{2} \left( \sum \limits_{i=1}^d \log a_i^2 \right) \int_{S^{d-1}} v_1^2 \ d\mu(v) 
        \geq 0. 
    \end{align*}
\end{proof}

A similar result can be found in \cite{ledrappier2003random} Proposition 2.2 and a sharper result in the $2\times 2$ case in \cite{AvilaBochiFormula}. The existence of Dedieu--Shub measures provide a strategy to prove Dedieu--Shub inequalties.

\subsubsection{} 

For example in the case of the original inequality, Theorem \ref{DS measures} implies Theorem \ref{originalineq}. 

\begin{proof}[Proof of Theorem \ref{originalineq} assuming \ref{DS measures}] \label{proof1.3}
    By Lemma \ref{projection}, there exists a Dedieu--Shub measure for $G = \GL(d,\C), K = \SU(d), X = \Grass(d,k)$ for every $k \in \{1,...,d-1\}.$ Let $f$ be monotone increasing. Denote the Dedieu--Shub measure by $\tilde{m}$. Given $k$ linearly independent vectors, let $$\phi(w_1,...,w_k) = [\langle w_1 \rangle,...,\langle w_1,...,w_k \rangle] \in \FF_d.$$ Recall that $v_i: GL(d,\R) \to \C^d$ denotes the $i$th eigenvector. Then \begin{align*}
        \int_{\U(d)} f \prod \limits_{i=1}^k |\lambda_i(UA)| \ d\nu(U) &= \int_{\U(d)} f \det A| \phi(v_1(UA),...,v_k(UA))  \ d\nu(U) \\
        &= \int_{\U(d)} \int_{\Grass(d,k)}f \det A| g_k \ d \delta_{\phi(v_1(UA),...,v_k(UA))}  \ d\nu(U) \\
        & \geq \int_{\U(d)} \int_{\Grass(d,k)} f \det A| g_k \ d\tilde{m}_{UA}(g_k) \ d\nu(U) \\
        % & =  \int_{\Grass(d,k)} f \det A| g_k \ d\int_{\U(d)} \tilde{m}_{UA}(g_k) \ d\nu(U) \\ 
        &= \int_{\Grass(d,k)} f \det A| g_k \ d\mu(g_k). \ \ (\text{Lemma} \ \ref{fubini} \ \text{and} \ \text{Theorem} \ref{DS measures})
    \end{align*}
\end{proof}

\subsubsection{} We can also ask this question over $\GL(2,\R)$. In this setting, the result  is due to \cite{dedieu2003mike}. Related is the Herman--Avila--Bochi formula \cite{AvilaBochiFormula}. Interestingly, when in $\GL(2,\R)^+$, because the measures are physical, we get equality.

\begin{proposition}[\cite{dedieu2003mike}]
    Suppose $A \in \GL(2,\R)$. Then for increasing $f$, $$\int_{\theta \in [0,2\pi)} f \rho(R_\theta A) \ d\theta = \int_{v \in S^1} f\|Av\| d\mu(v)$$ where $R_\theta = \begin{bmatrix}
        \cos \theta & - \sin \theta \\ 
        \sin \theta & \cos \theta
    \end{bmatrix}$.
\end{proposition}

\begin{proof}
    Suppose $A \in \GL(2,\R)$. 
    If $\det A > 0$, then \begin{align*}
        \int_{0}^{2\pi} f \rho(R_\theta A) \ d\theta &= \int_{0}^{2\pi} f \lim \limits_{n \to \infty} \|(R_{\theta} A)^n\|^{1/n} \ d\theta  \\
        &= \int_{0}^{2\pi} \int_{v \in S^1 }f(\| A v\| ) d m_{R_\theta A}(v) \ \ \  d\theta \ \ (\text{because} \ m_{R_\theta A} \ \text{is physical})\\
        % &= \int_{v \in S_1} f\|Av\|  \ d \int_0^{2\pi} m_{R_\theta A}(v) \ d\theta \\
        &= \int_{v \in S^1} f\|Av\| \ d\mu(v)  \ \ (\text{Lemma} \ \ref{fubini})
    \end{align*}

    If $\det A < 0$, a computation per \ref{proof1.3} will produce an inequality. 
\end{proof}

\section{Proof of Theorem \ref{GL2R}}

\subsection{Proof strategy}
Producing a Dedieu--Shub measure over $GL(2,\R)$ acting on $\RP^1$ splits into two cases -- the positive determinant and negative determinant case. They are different because of Cayley--Hamilton theorem: a $2 \times 2$ matrix with positive determinant can have both real and imaginary eigenvalues, but a $2 \times 2$ matrix with negative determinant can only have real eigenvalues. Hence in the case of positive determinant, we can have elliptic, parabolic, and hyperbolic phenomenon, whereas in the negative determinant case, we can have only hyperbolic phenomenon. The main difference is that in the hyperbolic case, invariant measures are supported on (convex combinations of) Dirac masses, whereas in the parabolic and elliptic cases, invariant measures are fully supported. 

The positive determinant case is proven in \cite{pujals2006expanding} and \cite{ledrappier2003random}. We present the proof of \cite{pujals2006expanding}. For the negative determinant case, we explicitly compute densities and normalize. 

\subsubsection{}
An immediate simplification we can make is that we  only need to consider diagonal matrices $A$ in the proofs of Theorems \ref{DS measures},\ref{originalineq}, 
\ref{GL2R}. Any $d \times d$ matrix $A$ admits a \textit{singular value decomposition} of the form $A = U \Sigma V$, where $U$ and $V$ are unitary and $\Sigma$ is diagonal with values in $\R_{\geq 0}$. This simplifies much of our analysis for Dedieu--Shub measures. 

\begin{lemma}     \label{SVD}
    Let $A \in GL(d,\R \text{ or }\C)$ admit singular value decomposition $A = V \Sigma W$. Then \begin{enumerate}
        \item[1.] For every $k \in \{1,...,d\}$, $$\int_{U(d)} \delta_{\lambda_k(UA) } \ d\nu(U) = \int_{U(d)} \delta_{\lambda_k(U\Sigma)} \ d\nu(U);$$ 
        \item[2.] For every $k \in \{1,...,d\}$ $$\int_{\Grass(d,k)} \delta_{\det A|g_k} \ d\mu(g_k) = \int_{\Grass(d,k} \delta_{\det \Sigma|g_k} \ d\mu(g_k). $$ 
    \end{enumerate}
\end{lemma}

\begin{proof}
    In case 1, let $A = V \Sigma W$ be the singular value decomposition and let $U \in U(d)$ be Haar random. We have that $UA = UV\Sigma W = WUV \Sigma$. Haar measure is translationally invariant, and so $WUV$ is also Haar distributed. In case 2, $\mu$ on $\text{Grass}(k,d)$ is by definition, stationary under $U(d)$ action and so $W g_k V$ is distributed in $g_k$. In the real case, under $O(d)$ action, the situation is the same. 
\end{proof}

Hence, throughout our paper when considering Dedieu--Shub measures, it is enough for us to consider diagonal matrices with non negative matrices, essentially presupposing our matrix was processed first through a singular value decomposition machine. Secondly, we can always prove results for matrices with determinany $\pm 1$. This is because we can divide out by the determinant. 
 
\subsection{Facts about the projective action}

\subsubsection{} Let $B  \in \GL(2,R)$. If $|\trace B| < 2$, then $B$ is called \textit{elliptic}; if $|\trace B| = 2$, then $B$ is called \textit{parabolic} and; if  $|\trace B| > 2$, then $B$ is called \emph{hyperbolic}. It is clear that in $\SO(2)\mathrm{diag}(a,a^{-1}),$ parabolic matrices and elliptic matrices with rational rotation number is a measure zero phenomenon. 

\subsubsection{}

By definition, $\RP^1$ is the space of lines passing through the origin in $\R^2$. We can define the \emph{angular parametrization} of $\RP^1$ as 
\begin{equation}
\omega \in (-\pi/2,\pi/2] \leftrightarrow \R \begin{pmatrix} \cos \omega \\ \sin \omega \end{pmatrix}
\end{equation}
identifying the line by the angle it forms with the $x$ axis.  In addition, we can define the \emph{anti--slope parametrization} of $\RP^1$ as 
\begin{align*}
    s \in \overline{\R} &\leftrightarrow \R \begin{bmatrix}
        s \\ 1
    \end{bmatrix} \\
    s = \frac{x}{y} &\leftrightarrow \R\begin{bmatrix}
        x \\ y
    \end{bmatrix}.
\end{align*}
A matrix $B = \begin{bmatrix}
    a & b \\
    c & d
\end{bmatrix}\in \GL(2,\R)$ acts on the anti--slope parametrization of $\RP^1$ by \emph{Möbius transformations} $$A.z = \frac{az + b}{cz + d}.$$ The unique $SO(2)$--invariant measure on the real line (that is in anti--slope parametrization) acting via Möbius transformations is the \emph{Cauchy distribution} given by \begin{equation} \label{cauchy_dist}
	\frac{1}{\pi} \, \frac{1}{1+s^2}  \, ds \, .
\end{equation}  
% Given $\zeta \in \C$, let $\xi = \mathrm{Re}(\zeta)$, $\eta = \mathrm{Im}(\zeta)$. The \textit{Poisson kernel} is $$\kappa_\zeta 
%  \coloneqq \frac{1}{\pi} \, \frac{\eta}{(s-\xi)^2 + \eta^2}	 \ (s \in \R).$$ 
% For all $A \in \SL(2,\R)$ and $\zeta \in \H$, we have $A_* (\kappa_\zeta) = \kappa_{A(\zeta)}$. 

\subsubsection{} A key fact is that a matrix in $\GL(2,\R)$ acting by Möbius transforms on $\overline{\C}$, is that it preserves the upper half plane $\H$. A related fact is that a Möbius transform preserving the closed unit disk $\overline{\mathbb{D}}$ is of the form \begin{equation}
    b(z) = e^{i\theta} \frac{z - c}{1 - \overline{c}z}
\end{equation} 
where $c \in \C$ satisfies $|c| < 1.$ Such a $b(z)$ is called a \emph{Blaschke factor.} It is a standard fact that any Möbius transform preserving $\D$ can be written as a Blaschke factor. \

\subsubsection{} The \emph{Cayley transform} $\varphi \in \SL(2,\R)$ defined as $$C(z) = \frac{z - i}{z + i}$$ maps $\H$ to $\mathbb{D}$ and $\overline{\R}$ to $\T := \partial \D$. This necessarily relates the dynamics of an $A \in GL(2,\R)$ acting on $\RP^1$ to the dynamics of a Blaschke factor $b(z)$ acting on $\T$ by the following commuting diagram:

\[ \begin{tikzcd}
\H \arrow{r}{A} \arrow[swap]{d}{C} & \H \arrow{d}{C} \\%
\T \arrow{r}{b}& \T.
\end{tikzcd} 
\]

\begin{lemma} \label{equivalence}
    Fix $A = \mathrm{diag}(a,a^{-1})$.
   Then the dynamics of $R_\theta A$ is conjugate to $$b(z) = e^{i \theta} \frac{z + \tanh(\ln a)}{1 + \tanh(\ln a) z}$$ has dynamics conjugate to $R_\theta A$ acting on $\RP^1$.
\end{lemma}

\begin{proof}

Conjugate $R_\theta$ and $A$ separately by the Cayley transform gives $C R_\theta C^{-1}.z = e^{i \theta} z$ and \begin{align*}
    b(z) &= C \circ A \circ C^{-1}(z) 
    = \frac{1}{2i} \begin{bmatrix}
        1 & -i\\ 1 & i
    \end{bmatrix} \begin{bmatrix}
        a & 0 \\ 0 & a^{-1}
    \end{bmatrix} \begin{bmatrix}
        i & i\\ -1 & 1
    \end{bmatrix}.z  = \begin{bmatrix}
        \frac{a + a^{-1}}{2} & \frac{a - a^{-1}}{2} \\
        \frac{a - a^{-1}}{2} & \frac{a + a^{-1}}{2}
    \end{bmatrix}.z \\ &= \begin{bmatrix}
        \cosh(\ln a) & \sinh (\ln a) \\
        \sinh (\ln a) & \cosh ( \ln a)
    \end{bmatrix}.z
   =  \frac{z+\tanh (\ln a)}{1 + \tanh (\ln a) z}.
\end{align*} Composing the two gives the result. 

\end{proof}

\subsection{Positive determinant}
The argument presented here follows \cite{pujals2006expanding}. 

\subsubsection{}
A \emph{Blaschke product}, $B:\C \to \C$, given by $$B(z) = e^{it}  \prod \limits_{i=1}^n \frac{z - a_i}{1- z \overline{a_i}}$$ where $n \geq 1$ $t \in [0,2\pi)$, $a_i \in \mathbb{C}, |a_i| < 1$ for $i \in [1,...,n]$. If $n = 1$, we call $B$ a \emph{Blaschke factor}. A Blaschke product is analytic in a neighborhood of the disk $\mathbb{D}$, and preserves the boundary $\partial \mathbb{D} = \mathbb{T}$. Denote Lebesgue measure on $\T$ by $\mu$. 
    % Let's prove it for a single factor $$g(z) = e^{it} \frac{z - a}{1-z\overline{a}}.$$ We want to show $|g(e^{i\varphi})|=1$. To this end, we have that \begin{align*}
    %     |g(e^{i\varphi})|^2 &= \left|\frac{(e^{i\varphi} - a)}{1 - e^{i\varphi} \overline{a}}\right|^2 \\
    %     &= \left(\frac{(e^{i\varphi} - a)(\overline{e^{i\varphi}} - \overline{a})}{(1 - e^{i\varphi} \overline{a})(1 - \overline{e^{i\varphi}} \overline{a})}\right) \\
    %     &= \frac{1 + a\overline{a}}{1 + a \overline{a}}\\
    %     &= 1. 
    % \end{align*}
    % Observe that we did not use the fact that $|a| < 1$. 

\subsubsection{}
To prove \textit{(1)} and   \textit{(2)} of Theorem \ref{GL2R}, by Lemma \ref{equivalence}, it is sufficient to prove the result for degree 1 Blaschke products. Recall that a \emph{physical measure} is simply the pushforward of Lebesgue. 

\begin{theorem}[\cite{pujals2006expanding}] \label{Blaschke}
    Fix a Blaschke product $B(z)$ of degree $n \geq 1$. Consider the family of circle maps acting on the boundary $\T :=  \partial \mathbb{D} \subset \C$ defined as $$\cF_B = \{B_\theta := R_\theta B(z): \theta \in [0,2\pi] \},$$ where $R_\theta = e^{i\theta}$. The physical measure of any $B_\theta \in \cF_B$ is either a Dirac mass or a smooth absolutely continuous measure. In either case, the physical measures, which we denote as $\nu_\theta$, are Dedieu -- Shub measures. 
\end{theorem}

\subsubsection{What is the physical measure of a $B_\theta$?}

\begin{lemma}
      Fix a Blaschke product $B(z)$ of degree $n \geq 1$. Then for a full measure set of $\theta$, there is a fixed point $\alpha(\theta) \in \overline{\D}$ of $B_\theta = e^{i\theta}B(z)$ such that $\alpha(\theta)$ is a sink. 
\end{lemma}

\begin{proof}
    
\end{proof}

The physical measure  of a given $B_\theta$ is $$\nu_\theta = \lim \limits_{n \to \infty} B_{\theta_*}^n(\mu) = \lim \limits_{n \to \infty} \sum \limits_{i=1}^n \delta_{B_\theta(x)}$$ for almost every $x \in \T$.

% For a full measure set of $\theta$, there are two possibilities of asymptotic dynamical phenomenon for orbits of $x \in \D$. Namely, either $B_\theta^n(x) \to \alpha(\theta) \in \T$, in which case $\nu_\theta$ is a Dirac mass, or, $B_\theta^n(x) \to \alpha(\theta) \in \D$, in which case $\frac{d\nu_\theta(e^{it})}{dt} = \cP(\alpha(\theta), e^{it})$, where $\cP(z,w) = \frac{1 - |z|^2}{|w-z|^2} $ is the Possion kernel. 

Recall that given a continuous function $h:\T \to \R$, with harmonic extension $\tilde{h}:\D \to \R$, the \emph{Poisson formula} gives that $$\tilde{h}(z) = \frac{1}{2\pi}\int_{0}^{2\pi} h(e^{it}) \cP(z,e^{it}) \ d\mu(t)$$ where $\cP(z,w) = \frac{1 - |z|^2}{|w-z|^2}$ is the \emph{Poisson kernel} \cite{lyons2017probability}. 

Then, because $\cP$ is equivariant, we have $$\int_\T h(z) d \left(\lim \limits_{n \to \infty} B^n_{\theta_*}(\mu)\right)(z) = \int_\T h(z) \ d\nu_\theta(z) = \int_\T h(z) \cP(\alpha(\theta), z) \ d\mu(z) = \tilde{h}(\alpha(\theta)).$$

% To compute it, firstly, given a continuous function $h: \overline{\D} \to \overline{\D}$ that is harmonic on $\D$, its value at $z \in \D$ is determined by its values on the boundary by the \emph{Poisson formula}
% $$h(z) = \frac{1}{2\pi} \int_0^{2\pi} \cP(z,e^{i\theta}) e^{i\theta} \ d\theta,$$ where $\cP(z,w) = \frac{1 - |z|^2}{|w-z|^2}$ is the \emph{Poisson kernel}. 

For a Blaschke factor $b(z)$, and defining $b_\theta = e^{i\theta} b(z)$, there is always a fixed point $\alpha(\theta)$ inside $\overline{\D}$ and if it is hyperbolic, the fixed point is on $\T$. The physical measure of $b_\theta$, denoted by $\mu_\theta$, is determined by $\alpha(\theta)$ and the Poisson formula. Namely, for $A \in \cB(\T)$, $$\nu_\theta(A) = \int_\T \mathbbm{1}_A(z) \cP(\alpha(\theta), z) d\mu(z).$$ In particular, the Radon--Nikodym derivative of $\nu_\theta$, when $\alpha(\theta) \in \D$, with respect to $\mu$ -- which is the Lebesgue measure on $\T$ -- is $$\frac{d\nu_\theta}{d\mu}(z) = \cP(\alpha_\theta,z).$$ 

For general Blaschke products, a similar result is true. 

\begin{proposition}[\cite{pujals2006expanding}]
  Furthermore, the physical measure of $B_\theta$ has density given by $$\frac{d \nu_\theta}{d\mu}(z) = \cP(\alpha_\theta,z)$$ where $\mu$ is Lebesgue measure on $\T$. What this means is that for every continuous $h:\D \to \R$, $$\int h \ d\nu_\theta = \int h \cP(\alpha_\theta,z) d\mu $$ Finally, $\alpha(\theta)$ varies continuously on $\D$ for a full measure set of $\theta.$
\end{proposition}

\begin{proof}
    Proposition 2.1, 2.2, 2.3 of \cite{pujals2006expanding}. 
\end{proof}

\subsubsection{Finishing the argument}

We wish to argue the following. 

\begin{proposition}[Proposition 4.3 of \cite{pujals2006expanding}]
    For all $n$,  $$\int_0^{2\pi} B^n_{\theta_*}(\mu) d\theta = \mu.$$
\end{proposition}

\begin{proof}
   It is sufficient to show for all continuous $h:\T \to \R$ that $$\int_{z \in \T} h(z) d\left( \int_{\theta \in [0,2\pi]} B^n_{\theta_*}(\mu) d\theta \right) (z)= \int_{z \in \T} h(z) \ d\mu(z).$$

   Now \begin{align*}
      \int_{z \in \T} h(z) d\left( \int_{\theta \in [0,2\pi]} B^n_{\theta_*}(\mu) d\theta \right) (z) 
       &=\int_{\theta \in [0,2\pi]}  \int_{z \in \T} h(z) d B^n_{\theta_*}(\mu)  (z) d\theta
       \\
       &= \int_{\theta \in [0,2\pi]} \tilde{h}(B^n_{\theta}(0))  \ d\theta \\
       % &= \int_{\theta \in [0,2\pi]} \frac{1}{2\pi}\int_{0}^{2\pi} h(e^{it}) \cP(B^n_{\theta}(0),e^{it}) \ dt   d\theta 
       &= \tilde{h}(B^n(0)) \\ 
       &= \tilde{h}(0) \ \ \text{(because } \tilde{h} \text{ is analytic in } \theta )\\
       &= \int_{z \in \T} h(z) \ d\mu(z).
   \end{align*}

\end{proof}

\begin{proof}[Proof of Theorem \ref{GL2R} parts \textit{(1), (2)}]
   By dominated convergence theorem, and because the previous proposition holds, we have that $$\lim \limits_{n\to \infty} \int_0^{2\pi} B^n_{\theta_*}(\mu) d\theta = \mu.$$
\end{proof}

\subsection{Negative determinant}

It remains to study the case when our diagonal matrix has form  $A = \begin{bmatrix}
    a & 0 \\
    0 & -a ^ {-1}
\end{bmatrix}$.

Parametrize $SO(2)$ by $\theta \in [0,2\pi)$ equipped with Lebesgue-Haar.  Consider $$ \cC = SO(2)^-A 
= \begin{bmatrix}
    a\cos \theta & a^{-1} \sin \theta \\
    a\sin \theta & -a^{-1}\cos \theta
\end{bmatrix}, \ \theta \in [0,2\pi]$$
Matrices in the coset $\cC \subset\SL(2,\mathbb{R})$, have negative determinant equal to $-1$, and hence have two real eigenvalues. 

For a fixed $R_\theta A$ with corresponding top eigenvector $v_\theta =(\cos \varphi, \sin \varphi) $. Observe that $$\|R_\theta A v_\theta\|^2 = \|Av_\theta\|^2= |\lambda|^2 = a^2 \cos^2 \varphi + a^{-2} \sin^2 \varphi. $$  We are interested in computing the distribution of $\varphi$. 

\begin{lemma}
    The cumulative density function of the trace is given by $$F_T(t) =  \frac{1}{\pi } \arccos\left( \frac{t}{a^{-1}-a}\right)$$ with probability density $$f_T(t) =  \frac{1}{(a-a^{-1})\pi\sqrt{1 - \left(\frac{t}{a^{-1}-a}\right)^2}}$$
\end{lemma}

\begin{proof}
    The trace of a matrix in $\cC$ is given by $(a - a^{-1}) \cos \theta$ for $\theta$ uniformly distributed in $[0,2\pi]$. Denoting the trace as $T$, we have \begin{align*}
        F_{T}(t) &= \P(T \leq t) \\
        &= \P((a-a^{-1})\cos(\theta) \leq t) \\
        &= \P\left(\theta \leq \arccos\left( \frac{t}{a^{-1}-a}\right)\right) \\
        &= \frac{1}{\pi } \arccos\left( \frac{t}{a^{-1}-a}\right).
    \end{align*}
\end{proof}

The eigenvalues of a matrix in $\cC$ in terms of the trace are $$\lambda_{\pm}(\theta) = \frac{t_\theta \pm \sqrt{t_\theta^2 + 4}}{2}.$$ Observe that $|\lambda_+| < |\lambda_-|$ precisely when $\theta \in [-\pi/2,\pi/2]$ and that the probability density of $\lambda_+$ is precisely the probability density of $-\lambda_-$. 

\begin{proposition}
    Denote $X$ as the random variable $\rho(R_\theta A)$ where $\theta$ is uniformly distributed in $[0,2\pi]$. Then the cumulative density function of $X$ is $$\frac{2}{\pi} \arccos \left( \frac{1-\rho^2}{(a^{-1} - a) \rho}\right)$$ and probability density function given by $$f_X(\rho) = \frac{2}{\pi} \frac{(\rho^2 + 1)}{\rho \sqrt{(a^{-2}-\rho^2)(\rho^2 - a^2)}}.$$ 
\end{proposition}

\begin{proof}
    We have that
    \begin{align*}
        F_{\lambda_+}(\rho) &= \P(\lambda_+ \leq \rho) \\
        &= \P\left( \frac{t + \sqrt{t^2 + 4}}{2} \leq \rho \right) \\
        &= F_T \left( \frac{\rho^2 - 1}{\rho}\right) \\
        &= \frac{1}{\pi} \arccos\left(\frac{1-\rho^2}{(a^{-1} - a)\rho}\right)
    \end{align*}

    By symmetry of $\lambda_+, \lambda_-$, we then get that \begin{align*}
        F_X(\rho) &= \frac{2}{\pi} \arccos \left( \frac{1-\rho^2}{(a^{-1} - a)\rho}\right) \\ 
        f_X(\rho) &= \frac{2}{\pi} \frac{a(\rho^2 + 1)}{\rho\sqrt{(1-a^2\rho^2)(\rho^2-a^2)}} \\ 
        &= \frac{2}{\pi} \frac{\rho^2 + 1}{\rho\sqrt{(a^{-2}-\rho^2)(\rho^2-a^2)}}
    \end{align*} 
\end{proof}

\begin{proof}[Proof of Theorem \ref{GL2R} (3)]
    We have $$\rho^2 = a^2 \cos^2 \varphi + a^{-2} \sin^2 \varphi$$ where $\varphi$ is the coordinate in $\RP^1 = [-\pi/2,\pi/2)$ of the corresponding eigenvector. This reduces to $$\rho^2 = (a^2 - a^{-2})\cos^2\varphi + a^{-2} \implies \varphi = \arccos\left( \sqrt{\frac{\rho^2-a^{-2}}{a^2 - a^{-2}}}\right).$$
    A computation shows $$\frac{d\varphi}{d\rho} =-\frac{\rho}{\sqrt{( \rho^2-a^2)(\rho^2 - a^{-2})}} \implies d\varphi = -\frac{\rho}{\sqrt{( \rho^2-a^2)(\rho^2 - a^{-2})}} d\rho.$$
    Substituting into the density of $\rho$ and normalizing by the mass of $\RP^1$, we get $$\frac{\rho + \rho^{-1}}{\rho} d\varphi. $$
    Inverting gives us the Dedieu--Shub weight. 
\end{proof}

A similar calculation shows that the Dedieu--Shub measure for $\GL(2,\C)$ are those in \ref{dim2weights}. Namely, given a Haar random $U \in SU(2)$, the probability density function of $\rho(U\mathrm{diag}(a,a^{-1}))$ is $$f(\rho) = \frac{2}{a^2 - a^{-2}}(\rho + \rho^{-3}) \ d\rho$$. Now given the corresponding eigenvector $(X,Y) \in S^3 \subset \C^2$, it necessarily satisfies $\rho^2 = a^2|X|^2 + a^{-2}|Y|^2 = (a^2 - a^{-2})X^2 + a^{-2}$. Furthermore, the expansion rate of an eigenvector in $\CP^1 = S^3$ is uniquely determined by its \emph{height} $|X|$ and each fiber of the map $S^3 \to |X|$ (the Hopf fibration) has constant mass. So, because the chain rule give that $2\rho \ d\rho = (a^2 - a^{-2}) \ dX,$  substituting gives that \begin{equation}
    \frac{2}{a^2 - a^{-2}}(\rho + \rho^{-3}) \ d\rho = 1 + \rho^{-4} \ dX = \frac{\rho^2 + \rho^{-2}}{\rho^2} \ dX.
\end{equation} Dividing gives \ref{dim2weights}.

% \begin{proposition}
%     The probability density function of $\|Av\|$ for $v$ uniformly distributed in $S^1$ is given by $$f_Y(y) = \frac{2}{\pi} \frac{y}{\sqrt{(a^2 - a^{-2})^{-1}(y^2 - a^{-2})} \sqrt{1- (a^2 - a^{-2})^{-1}(y^2 - a^{-2})}}.$$
% \end{proposition}

% \begin{proof}
%      Let $Y = \|Av_\varphi\|$ be a random variable where $v_\varphi = (\cos \varphi, \sin \varphi)$ and $\varphi \in [0,2\pi]$ is equidistributed. Then $Y^2 = (a^2 - a^{-2})\cos^2(\varphi) + a^{-2}. $ The cumulative density function of $Z = \cos^2 \varphi$ where $\varphi$ is uniformly distributed in $[0,2\pi]$ is $$F_Z(z) = \frac{\pi - \arccos(\sqrt{z})}{\pi}.$$
%      Change of variables and differntiating gives us that the probability density function is given by $$f_Y(y) = \frac{2}{\pi} \frac{y}{\sqrt{(a^2 - a^{-2})^{-1}(y^2 - a^{-2})} \sqrt{1- (a^2 - a^{-2})^{-1}(y^2 - a^{-2})}}.$$
% \end{proof}

\section{Proof of Theorem \ref{DS measures}}

\begin{figure}
    \centering
    \includegraphics[width=\textwidth]{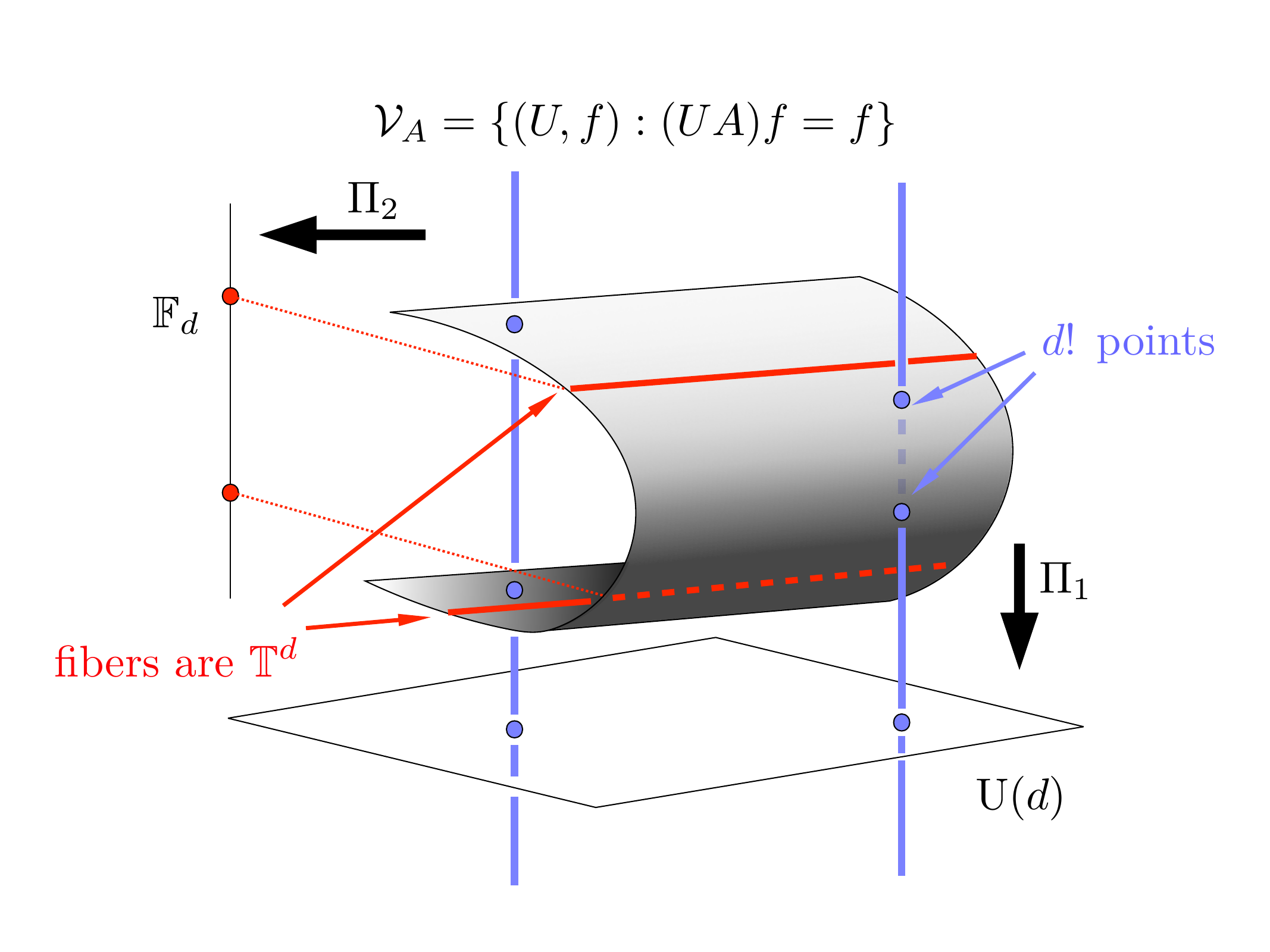}
    \caption{}
    \label{fig:fibration}
\end{figure}

The putative Dedieu--Shub measure $m:\GL(d,\C) \to \Prob(\FF_d)$ is given as follows. Let $p:\Sym(d) \times \mathrm{GL}(d,\C) \to [0,1]$ be defined as, given $\sigma  \in \Sym(d)$ $$p_\sigma(A) = \frac{\prod \limits_{j = 1}^d |\lambda_j|^{2(d - \sigma(j))}}{\sum \limits_{\pi \in \Sym(d)} \prod \limits_{j=1}^d |\lambda_j|^{2(d - \pi(j))}}.$$
Define $m: \GL(d,\C) \to \Prob(\FF_d)$ as $$m_A = \sum \limits_{\sigma \in \Sym(d)} p_\sigma(A) \delta_{\phi(v_{\sigma(1)}(A),...,v_{\sigma(d)}(A))}.$$

The measure $m_A$ is necessarily $A$ invariant and the function $m(\cdot)$ is also measurable. So it remains to prove the Dedieu--Shub property. 

\subsection{}

We only need to prove the theorem for a fixed diagonal $A = \mathrm{diag}(a_1,...,a_d)$, where $a_i > 0$ by Lemma \ref{SVD}. Define the variety $$\cV_A = \{(U,f): (UA)f = f\} \subset \U(d) \times \FF_d.$$ Denote the induced Riemannian volume as $d\cV_A$. If $M \subset N$ are both Riemannian manifolds, we denote by $dM|_N$, the induced Riemannian volume of $N$ to $M$. So, for example, if $S$ is a submanifold of $\cV_A$, let $d\cV_A|_{S}$ denote a restriction of the volume. Define the projections $\Pi_1: \V_A \to \U(d)$ and $\Pi_2: \V_A \to \FF_d$. 

\subsection{Fibers of $\Pi_1$ and $\Pi_2$}
\subsubsection{Fibers of $\Pi_1$} We have the following.
\begin{lemma}[\cite{dedieu2003mike} Proposition 5] \label{distinct_eigenvalues}
    Let $A \in \GL(d,\C)$. Then the set $$G := \{U \in \U(d): UA \text{ has eigenvalues of distinct modulus } \}$$ has Haar measure $1.$
\end{lemma}

\begin{proof}
    See \cite{dedieu2003mike} Proposition 5. To sketch it, we wish to show that $G$ has full Haar measure and we do this by showing the complement has zero Haar measure. Consider the smooth variety $$S = \{ (U,V) \in U(d) \times U(d) : V \text{ diagonalizes } UA\}.$$ In particular, if $(U,V) \in S$, then $VUAV^* = \mathrm{diag}(\lambda_1,...,\lambda_n)$. Define the evaluation maps 
        $$ev_1(U,V) = VUAV^*, \ \ \ ev_2^{i,j}(\mathrm{diag}(\lambda_1,...,\lambda_d) = |\lambda_i| - |\lambda_j|.$$
 The composition $ev_2^{i,j} \circ ev_1 : S \to \R$ is real analytic and non constant. Now a real analytic function can only have a positive measure set of roots if and only if it is identically zero. This is simply because almost every point in the domain is a density point, so the derivative would have to be zero everywhere if not. So the union of the zero set of $ev_2^{i,j} \circ ev_1$, over every pair $i,j$ has measure zero. This necessarily implies the set of matrices in the domain of $ev_2^{i,j} \circ ev_1$ with eigenvalues of repeat modulus has measure zero. 
\end{proof}

Abusing notation, whenever we integrate over $U(d)$, we will really mean integrating over $G$ as in Lemma \ref{distinct_eigenvalues}.

\subsubsection{Fibers of $\Pi_2$}
Any $d \times d$ matrix $A$ admits a $QR$--factorization $A = QR$, where $Q$ is unitary and $R$ is upper triangular. The interpretation is that the columns of $Q$ induces a flag fixed by $A$. There is a related factorization, the Schur factorization $A = QRQ^{-1}$, which is produced by $QR$--factorization. This is useful for our purposes due to the following lemma.

\begin{lemma}[Dedieu--Shub \cite{dedieu2003mike} Proposition 1]
\label{qr}

    The stabilizer of a canonical flag in $U(d)$ is the torus group.
    
\end{lemma}

\begin{proof}
    Let $\T_d = \{ \mathrm{diag}(e^{i\theta_1},...,e^{i\theta_d})\}$. Let $f_0$ be a canonical flag. If for some $V \in U(d)$, we have $Vf_0 = f_0$, then $V$ is in $\T^d$ -- every other matrix would rotate the frame. 
    % Fix $A$ diagonal and a $U \in U(d)$. Suppose $UA$ admits the Schur-factorization $UA = QRQ^{-1}$ where $Q$ is unitary and $R$ is upper triangular. In particular, $UA$ fixes the flag $\phi(q_1,...,q_d)$, where $q_i$ the $i$th column of $Q$. Then $UA$ fixes the flags $$\{\phi(Q \mathrm{diag}(e^{i\theta_1},...,e^{i\theta_d})Q^{-1}): \theta_i \in [0,2\pi) \} \subset \FF_d$$ recalling that $\phi:\GL(d) \to \FF_d$ is $$\phi(M) = [\langle M_1 \rangle \subset \dots \langle M_1,\dots,M_d \rangle ]$$ In particular, the modulus of the eigenvalues of $Q^{-1} \mathrm{diag} (e^{i\theta_1},...,e^{i\theta_d})Q UA $ are the same as that of $UA$. 
    
    % By Schur factorization, $QUAQ^{-1}$ is upper triangular. Multiplying by a diagonal matrix of unit modulus complex numbers does not change the absolute values of the spectra. Neither does conjugation. 
\end{proof}
\subsubsection{Putting it together}
\begin{proposition} \label{summary} Fix $A$. We have the following:
    \begin{enumerate}
        \item For Haar almost every $U \in U(d)$, $\# \Pi_1^{-1}(U) = d!$ composed of all $d!$ permutation of distinct eigenvectors of $UA$ and
        \item Recall $\T_d = \{ \mathrm{diag}(e^{i\theta_1},...,e^{i\theta_d})\}$. For every $f \in \FF_d$, $\Pi_2^{-1} f$ is $U_2 \cdot \T^d \cdot U_1$ for some $U_1, U_2 \in \U(d).$
    \end{enumerate}
\end{proposition}

\begin{proof}
    Part 1 follows from Lemma  \ref{distinct_eigenvalues} -- for almost every $U$, $UA$ has distinct eigenvalues and hence $d$ distinct eigenvectors. All $d!$ flags formed by the distinct eigenvectors are fixed. Part 2 follows from Lemma \ref{qr} and the following argument.  Let $f_0$ denote a canonical flag. Given a flag $f$, note that $$\{U:UAf = f\} = \{U_2VU_1:V \in \T^d\} = U_2\T^d U_1, $$ where $U_1$ and $U_2$ are picked such that $U_1(Af) = f_0$ and $U_2(f_0) =f.$
\end{proof} 

\subsection{Strategy} \label{discussion}

 Recall our notation that $$\phi(w_1,...,w_d) = [\langle w_1 \rangle \subset \langle w_1, w_2 \rangle  \subset ... \subset \langle w_1, ..., w_d \rangle].$$ For a given $A \in \GL(d,\C)$, let 
\begin{align*}
\lambda_i(A) &= \text{ the ith largest in modulus eigenvalue of } A\\
v_i(A) &= \text{the eigenvector corresponding to } \lambda_i(A).
\end{align*}

 We want to show the result in the weak star topology. In particular, for fixed $A \in \GL(d,\C)$ and every continuous $h:\FF_d \to \R$, we want to show \begin{align*}
     \int_{\FF_d} h(f) \ d\mu(f) 
 %&= \int_{f \in \FF_d }h(f)  \int_{U(d)} \ dm_{UA}(f) \ d\nu(U) \\
  &= \int_{U(d)} \int_{f \in \FF_d }h(f) \ dm_{UA}(f) \ d\nu(U) \\ 
  &= \int_{U(d)} \sum_{\sigma \in \Sym(d)}  h(\phi(v_{\sigma(1)}(UA),...,v_{\sigma(d)}(UA)) \frac{\prod \limits_{j = 1}^d |\lambda_j(UA)|
 ^{2(d - \sigma(j)))}}{\sum \limits_{\pi \in \Sym(d)} \prod \limits_{j=1}^d |\lambda_j(UA)|^{2(d - \pi(j))}} \ d\nu(U) \\ 
 \end{align*}

 The co--area formula allows us a method to do this. 
\subsection{Coarea formula}

For details, see \cite{evans2018measure}.
 
\begin{definition}[Linear normal jacobian]
    
Let $X$ and $Y$ be two finite dimensional inner product spaces. Let $A:X \to Y$ be linear and real. The \textit{normal jacobian} of $A$ is \begin{align*}
    NJ(A) & := \det (AA^*)^{1/2} \\ 
    & = |\det(A|_{\ker(A)^\perp})|\ \ \text{(explaining the use of the adjective ``normal.'')}
\end{align*}

Observe that $NJ(A) \geq 0$ with strict inequality if and only if $A$ has no zero singular values if and only if $A$ is surjective.
\end{definition}

\begin{definition}[Nonlinear normal jacobian] \label{nonlinearNJ}
    
Let $X$ and $Y$ be two real Riemannian manifolds. The \textit{normal jacobian} of a map $F:X \to Y \in C^1$ is $$NJ(F)(x) = \det(DF(x) DF(x)^*)^{1/2} \geq 0.$$  By the same observation as before, the inequality is strict if and only if $x$ is a regular point of $F$. 
\end{definition}

\begin{remark}
\label{remarkable}
If $A$ in the definition of a linear normal jacobian is a complex matrix, then as $A$ induces a map of $A_\R : \R^{2d} \to \R^{2d}$, we have that $$|\det A_\R| = |\det A|^2.$$ Hence if $X$ and $Y$ are \emph{complex} Riemannian manifolds, we have that $$NJ(F)(x) = |\det DF(x)|_{H_x}|^2 = \det(DF(x)DF(x)^*). $$
\end{remark}

\begin{definition} \label{almost_submersions}
    We say a map $F:X \to Y$ between two Riemannian manifolds is an \emph{almost submersive} if \begin{itemize}
        \item $F:X \to Y$ is smooth and surjective and 
        \item $DF(x):T_x X \to T_{F(x)}Y$ surjective for almost all $x \in X$. 
    \end{itemize}
\end{definition}

\begin{theorem}[Coarea formula]
    \label{coarea}
    Let $X$ and $Y$ be real Riemannian manifolds with volume forms $dX$ and $dY$. Suppose $F:X \to Y$ is an almost submersion. Denote the fibers of $F$ as $F^{-1}(y)$, for $y \in Y$. 
    
    Then, for any integrable $h:X \to \R$, 
    \begin{align*}
        \int_{x \in X} h(x) \ dX(x) &= \int_{y \in Y} \int_{x \in F^{-1}(y)} \frac{h(x)}{NJ(F(x))} \ dX|_{F^{-1}(y)}(x) \ dY(y), 
    \end{align*}
\end{theorem}

\begin{remark}
For complex manifolds, $$\int_{x \in X} h(x) \ dX(x) = \int_{y \in Y} \int_{x \in F^{-1}(y)} \frac{h(x)}{\det (DF(x) DF(x)^*)} \ dX|_{F^{-1}(y)}(x) \ dY(y).$$ 
\end{remark}

\begin{remark}
    If $N$ is a submanifold of $Y$, then $dY|_N$ is the induced volume measure on $N$. If $\dim N = 0$, the measure $dY|_N$ is the counting measure. Also, if $\dim X = \dim Y$ the the inner integral is a sum because fibers are discrete, and the measure on fibers is simply the counting measure. 
\end{remark}

\subsection{The double fibration trick}
Dedieu and Shub's double fibration trick, see also \cite{kodat2024average} and \cite{armentanorandom} for other usage, is this.

Let $Y_1$, $Y_2$, and $X$ be Riemannian manifolds and suppose. Let $dY_1,dY_2$ and $dX$ denote volumes. Suppose you are given almost submersions (Definition \ref{almost_submersions}) $F_1:X \to Y_1$ and $F_2: X \to Y_2$. 

\begin{enumerate}
    \item Start with $h: Y_2 \to \R$.
    \item Define $\tilde{h} : X \to \R$ by $$\tilde{h}(x) = \frac{NJ(F_2)(x)}{\mathrm{Vol} (F_2^{-1}(F_2(x))}h(F_2(x))$$
    \item The coarea formula with respect to $F_2$ gives that $$\int_X \tilde{h} \ dX = \int_{Y_2} h \ dY_2. $$
    \item The coarea formula with respect to $F_1$ gives that $$\int_X \tilde{h} \ dX = \int_{y_1 \in Y_1} \int_{F_1^{-1}(y_1)} \frac{\tilde{h}(x)}{NJ(F_1)(x)}  \ dX|_{F_1^{-1}(y_1)}(x)dY_1(y_1).$$
    \item Conclusion: We have $$\int_{Y_2} h \ dY_2 =  \int_{y_1 \in Y_1} \underbrace{\int_{F_1^{-1}(y_1)} \frac{NJ(F_2)(x)}{\mathrm{Vol}(F_2^{-1}(F_2(x)))NJ(F_1)(x)}  \ dX|_{F_1^{-1}(y_1)}(x)}_{\diamond}dY_1(y_1).$$
    The goal is now to evaluate $\diamond$.
\end{enumerate}

% The trick of Dedieu and Shub is like this. Consider a product manifold $V = X \times Y$, where $X$ and $Y$ are Riemannian. Let $F_1$ and $F_2$ be projections from $V$ to $X$ and $Y$ respectively. Let $h:Y \to \R$ continuous and bounded be given. Denote 
% % $\tilde{h}_1:V \to \R$ by $$\tilde{h}_1(x,y) = h(y) \frac{NJ(F_1)(x,y)}{\mathrm{Vol} \ F_1^{-1}(y) }.$$ 
% $\tilde{h}:V \to \R$ by $$\tilde{h}(x,y) = h(y) \frac{NJ(F_2)(x,y)}{\mathrm{Vol} \ F_2^{-1}(y) }.$$ 
% The coarea formula asserts that 
% \begin{align*}
%     \int_{(x,y) \in V} \tilde{h}(x,y) \ dV(x,y) &=  \int_{x \in X} \int_{(x,y)\in F^{-1}_1(x)} h(y) \frac{NJ(F_2)(x,y)}{\mathrm{Vol}\ F_2^{-1} (y)} dF_1^{-1}(x)(y) \ dX(x) \ \text{ (along } F_1 )\\
%      &=  \int_{y \in Y} \int_{(x,y)\in F^{-1}_2(x)} h(y) \frac{NJ(F_2)(x,y)}{\mathrm{Vol}\ F_2^{-1} (y)} dF_2^{-1}(y)(x) \ dY(y) \ \text{ (along } F_2 )\\
%      &= \int_{y\in Y} h(y) \ dY(y) \  \text{ (cancellation).}
% \end{align*}

% Now let's see what happens when  $Y$ is $\FF_d$ and $X$ is $\U(d)$.

We now apply the double fibration trick to our setting, where $\cV_A$, $\Pi_1,\Pi_2$ takes the place of $X, F_1,F_2$.

%  Now, fix continuous $h:\FF_d \to \R$. The coarea formula applied to the function $$(U,f) \in \V_A \to h(f) \frac{NJ(\Pi_2)(U,f)}{\mathrm{Vol} \ \Pi^{-1}_2(f) } $$ with respect to the function $\Pi_2:\V_A \to \FF_d$ (this takes the place of $F$ in Theorem \ref{coarea}) gives that 
% \begin{align*}
%     \int_{\V_A} h(f) \frac{NJ(\Pi_2)(U,f)}{\mathrm{Vol} \  \Pi^{-1}_2(f) } \ d\V_A(U,f) \ \ \ \ \ \ \ \ \ \ \ \ \ \ \ \ \ \ \ \ \ \ \ \ \ \ \ \ \ \ \ \ \ \ \ \ \ \ \ \ \ \ \ \ \ \ \ \ \ \ \ \ \ \ \ \ \ \ \ \ \ \  \ \ \ \ \ \ \ \  \ \ \ \ \ \  \ \ \ \ \ \ \ \  \\ =
%     \int_{f \in \FF_d} \int_{(U,f) \in \Pi_2^{-1}(f)} h(f) \frac{NJ(\Pi_2)(U,f)}{\mathrm{Vol} \ \Pi^{-1}_2(f) NJ(\Pi_2)(U,f)} \ d\Pi^{-1}_2(f)(U,f) \ d\mu(f) \\
%      \ \ \ \ \ \ \ \ \ \ \ \ \ \ \ \ \ \ \ \ \ \ \ \ \ \ \ \ \ \ \ \ \  \ \ \ \ \ \ \ \ 
%  = \int_{f \in \FF_d} h(f) \ d\mu(f)
% \end{align*}

% Applying the coarea formula to the same function with respect to $\Pi_1:\cV_A \to \U(d)$ gives 
% $$ \int_{\V_A} h(f) \frac{NJ(\Pi_2)(U,f)}{\mathrm{Vol} \ \Pi^{-1}_2(f) } \ d\V_A(U,f) = \int_{U \in \U(d)} \sum_{(U,f) \in \Pi_1^{-1}(U)} h(f) \frac{NJ(\Pi_2)(U,f)}{\mathrm{Vol} \ \Pi^{-1}_2(f) NJ(\Pi_1)(U,f)} \ d\nu(U).$$ 

% This written compactly is the following. 

\begin{proposition}[This is Theorems 17 and 23 in \cite{dedieu2003mike}] We have the following.

\begin{enumerate}
    \item Let $h:\FF_d \to \R$ be continuous. Then
    \begin{equation*}
        \int_{f \in \FF_d} h(f) \ d\mu(f) = \int_{U \in \U(d)} \sum_{(U,f) \in \Pi_1^{-1}(U)} h(f) \frac{NJ(\Pi_2)(U,f)}{\mathrm{Vol}\  \Pi^{-1}_2(f) NJ(\Pi_1)(U,f)} \ d\nu(U).
    \end{equation*}

    \item Let $g: U(d) \to \R$ be integrable. Then 
        \begin{equation*}
            \int_{U \in \U(d)} g(U) \ d\nu(U) =
        \int_{f \in \FF_d} \int_{(U,f) \in \Pi_2^{-1}(f)} g(U) \ \frac{NJ(\Pi_1)(U,f)}{\mathrm{Vol} \ \Pi_1^{-1}(U) NJ(\Pi_2)(U,f)} \ d\cV_A|_{\Pi_2^{-1}(f)} \ d\nu(U)
        \end{equation*}
\end{enumerate}
\end{proposition}

\begin{proof}
    We simply apply the double fibration trick twice. 
\end{proof}

\subsubsection{}

We now analyze all of the terms. First some definitions. 

Denote the evaluation map $\Phi_A:\U(d) \times \FF_d \to \FF_d$ as $$\Phi_A(U,f) = (UA)f$$ and its lift $\hat{\Phi}_A:\U(d) \times \FF_d \to \FF_d \times \FF_d$ as $$\hat{\Phi}_A(U,f) = (\Phi(U,f),f).$$ Then $$\cV_A = \hat{\Phi}^{-1}_A(\{(f,f) \in \FF_d \times \FF_d\}).$$ Denote the partial derivative of $\Phi_A$ along $U(d)$ and $\FF_d$ as $\cD_{U(d)} \Phi_A(U,f)$ and  $\cD_{\FF_d} \Phi_A(U,f)$ respectively. 

Secondly, given a $(U,f) \in \cV_A$, $f = \phi(v_1,...,v_d)$ where $v_1,...,v_n$ are eigenvectors of $UA$. If we presume these eigenvectors correspond to eigenvalues $|\lambda_1| > ... > |\lambda_d|$, we can index each flag fixed by a fixed $UA$ with elements in the symmetric group -- that is $UA$ also fixes $\phi(v_{\sigma(1)},...,v_{\sigma(d)})$. 

\begin{proposition} \label{thing}
    We have the following.

    \begin{enumerate}
        \item[a.] 
        Suppose $f = \phi(v_{\sigma(1)},...,v_{\sigma(d)}) \in \Pi_2(\cV_A)$ fixed by some $UA.$ We have that
        \begin{align*}
            NJ(\Pi_1)(U,f) &\overset{(1)}{=} |\det(\mathrm{id}_{T_f\FF_d} - \cD_{\FF_d}\Phi_A(U,f))| \\
            &\overset{(2)}{=} \prod \limits_{j < i} \left| 1 - \frac{\lambda_{\sigma(i)}}{\lambda_{\sigma(j)}}\right|.
        \end{align*}
        \item[b.] 
        \begin{align*}
            NJ(\Pi_2)(U,f) &\overset{(3)}{=} |\det (\cD_{\U(d)} \Phi_A(U,f) \cD_{\U(d)} \Phi_A(U,f) )^*| \\
            &\overset{(4)}{=}\mathrm{Vol}(\T^d)\\
            &\overset{(5)}{=} \mathrm{Vol}(\Pi_2^{-1}(f)) 
        \end{align*}
    \end{enumerate}
\end{proposition}

Applying these to the previous proposition, we get the following. 

\begin{corollary}
    We have the following:
    \begin{enumerate}
        \item Let $h:\FF_d \to \R$ be continuous. Then
        \begin{align}
            \int_{f \in \FF_d} h(f) \ d\mu(f) =
            \int_{U \in U(d)} \sum_{\sigma \in \Sym(d)} h(\phi(v_{\sigma(1)},...,v_{\sigma(d)})) \frac{1}{ \prod \limits_{j < i} \left| 1 - \frac{\lambda_{\sigma(i)}(UA)}{\lambda_{\sigma(j)}(UA)} \right|^{2}} \ d\nu(U).
    \label{almost}
        \end{align}
    \item  
     Let $g: U(d) \to \R$ be integrable. Then 
        \begin{align}
        \int_{U \in \U(d)} g(U) \ d\nu(U) =
        \int_{f \in \FF_d} \int_{(U,f) \in \Pi_2^{-1}(f)} g(U) \ \prod \limits_{j < i} \left| 1 - \frac{\lambda_{\sigma(i)}(UA)}{\lambda_{\sigma(j)}(UA)} \right|^{2} \ d\cV_A|_{\Pi_2^{-1}(f)} \ d\nu(U) \label{the equivalent}
        \end{align}
    \end{enumerate}
\end{corollary}

\begin{proof}[Proof of Proposition \ref{thing}]
We need to explan all five equalities. 

\noindent \textbf{Equality (1):} The map $\Pi_1:\cV_A \to \U(d)$ is an almost submersion. We wish to say, by Sard's theorem, that almost every point $U \in \U(d)$ is a regular value. An issue arises that the image under $\Pi_1^{-1}$ of the set $U \in U(d)$ such that $UA$ has repeated eigenvalues form a singularity set in $\cV_A$. To deal with this, we take the algebraic set of $$S = \{ (U,f): UAf = f \text{ and } UA \text{ has repeat modulus eigenvalues}\} \subset \cV_A.$$ It is a subvariety of lower dimension and hence is a set of $d\cV_A$--measure $0$. Therefore we can construct the blowup of $cV_A$ along $S$, denoted by $\text{BL}_S(\cV_A)$. In the blowup process, assuming  $S$ has dimension $k$, we replace it with a subvariety of $\CP^1 \times ... \times \CP^2$, where the product is taken $k$ times. (See \cite{voisin2003hodge} p. 75 section 3.3.3 for more details on why we can always do this.) Now we can apply Sard's theorem to $\tilde{\Pi}_1:\text{BL}_S(\cV_A) \setminus  S \to U(d) \setminus \Pi_1(S)$, because now $\Pi_1$ is biholomorphic and hence $C^\infty$. Blowing down, we can say that for the original $\Pi_1$, almost every point in $U(d)$ is a regular value. This argument is essentially the proof of Bertini's theorem from algebraic geometry. 

Now, a $U \in U(d)$ is a regular value for $\Pi_1$ if and only if for all $f \in \FF_d$ with $(U,f) \in \cV_A$, the map $\mathrm{id}_{T_f\mathbb{F}_d} - \cD_{\FF_d} \Phi_A(U,f)$ is invertible. This implies that the tangent space of $\cV_A$ at $(U,f)$ is given by $$T_{(U,f)}\cV_A = \{(\dot{U}, \dot{f}) \in T_U \U(d) \times T_f \FF_d: \dot{f} = (\mathrm{id}_{T_f\F_d} - \cD_{\F_d}\Phi_A(U,f))^{-1} \cD_{\U(d)} \Phi_A (U,f) \dot{U} \}.$$ So we conclude that $$NJ(\Pi_1)(U,f) = |\mathrm{id}_{T_f\F_d} - \cD_{\FF_d}\Phi_A(U,f)|.$$

\noindent \textbf{Equality (2):} The quantity $\cD_{\FF_d}\Phi_A(U,f)$ is the derivative of the attracting flag for the QR --algorithm, and QR -- factorization is a Morse Smale dynamical system, as explained in \cite{shub1987g}. First, we treat dimension $2$. Let $f$ be a fixed flag in $\FF_2 = \CP^1$ and let $f^\perp$ be orthogonal to $f$. With respect to the basis $f, f^\perp$,$$ UA = \begin{bmatrix}
    \lambda_1 & b \\
    0 & \lambda_2
\end{bmatrix},$$ and obviously $f = [1,0]^T$ and $f^\perp = [0,1]^T$. Let $\tilde{f}$ be a small perturbation $[1,\epsilon]^T$. Then $$UA\tilde{f} = \begin{bmatrix}
    \lambda_1 + b\epsilon \\
    \lambda_2 \epsilon \end{bmatrix}\overset{\text{parallel to}}{=} 
    \begin{bmatrix}
        1 \\
       \frac{\lambda_2 \epsilon}{\lambda_1 + b\epsilon}
    \end{bmatrix}
    \approx 
    \begin{bmatrix}
        1 \\
       \frac{\lambda_2 }{\lambda_1 }
    \end{bmatrix}.
$$
To get equality (2), we now simply induct this dimension $2$ proof. We consider a flag $f = \phi(v_{\sigma(1)},...,v_{\sigma(d)}) \in \FF_d$ such that some $UA$ fixes it and $v_i$ are eigenvectors corresponding to $|\lambda_1(UA)| > \dots > |\lambda_d(UA)|$. Let $ij \in \binom{d}{2}$. Let $$C_{i,j} = \{\phi(w_1,...,w_d): w_i \text{ are orthonormal and } w_k = v_k \text{ if } k \notin ij \}.$$ Along $C_{ij}$, we do the dimension $2$ proof and conclude that $$\cD_{\FF_d}\Phi_A(U,f)|_{C_{ij}} =\cD_{\FF_d}\Phi_A(U, \phi(v_{\sigma(1)},...,v_{\sigma(d)}))|_{C_{ij}}  = \frac{\lambda_\sigma(i)}{\lambda_\sigma(j)}.$$

\noindent \textbf{Equality (3): } 
By Definition \ref{nonlinearNJ}, the numerator is $$NJ(\Pi_2)(U,f) = |\det \cD_{U(d)} \Phi_A(U,f) \cD_{U(d)} \Phi_A(U,f)^*|,$$ noting that the map is complex. 

\noindent \textbf{Equality (4): } 
The idea is to argue that $NJ(\Pi_2)(\mathrm{Id}_d, \phi(e_1,...,e_d)) = \mathrm{Vol}(\T_d)$, and then to argue that if we act on the flag $\phi(e_1,\dots,e_d)$ by the unitary group, that the normal jacobian is preserved. 
Given an $f$, let $\Phi_f(U) = UV\T^d$ where $V$ is such that the flag $Af = \phi(V\T^d). $
We have that $\cD_{U(d)}\Phi_A(U,F) = \cD_{U(d)}\Phi_{f}(U).$ The normal jacobian of $\Phi_f$ is independent of $f$, $V$ and $U$ and equals $\mathrm{Vol} (\T^d)$ for the following reason. (The proof we recount is the content of Proposition 9 and Corollary 21 in \cite{dedieu2003mike}). First, let $V = \mathrm{Id}.$ Then $\Phi_f(U) = U\mathbb{T}^d$ is the projection from $\U_d \to \U_d/\T^d.$ The normal to the fiber is mapped isometrically to $\U_d / \T^d.$ Now $R_W:\U(d) \to \U(d)$  defined as $R_W(U) = U$ is an isometry of $\U(d)$ and the fibers of $\Phi_f$ are reciprocal images of $R_W$ of the fibers $\Phi_{\phi(\mathrm{Id})}$. 

\noindent \textbf{Equality (5)} This is Proposition \ref{summary}. 

\end{proof}

\subsection{Computing the coefficient} \label{computationally}

% We want to simplify the right hand side of Proposition \ref{almost}, in particular the coefficient involving eigenvalues. It is equivalent to simplifying the coeffiencient in  Proposition \ref{the equivalent}. 

Now, Lemmas \ref{summary} and \ref{qr} gives that $$\int_{(U,f) \in \Pi_2^{-1}(f)} \prod \limits_{j < i} \left| 1 - \frac{\lambda_{\sigma(i)}(UA)}{\lambda_{\sigma(j)}(UA)} \right|^{2} 
 d \cV_A|_{\Pi_2^{-1}(f)} = \int_{\T^d}  \prod \limits_{j < i} \left| 1 - \frac{\lambda_{\sigma(i)}(UA)}{\lambda_{\sigma(j)}(UA)} \right|^{2} \ d\nu_{\T^d}(U)$$ as each fiber of $\Pi_2$ is isomorphic to $\T^d.$ Now we compute the right hand side.We follow the analysis contained in the proof of Proposition 8 in \cite{dedieu2003mike}.

 Recall that the \emph{Van der Monde determinant} is $$V(\lambda_1,...,\lambda_n) = \det \begin{bmatrix}
     1 & \lambda_1 & \dots & \lambda_1^{n-1} \\ 
     1 & \lambda_2 & \dots & \lambda_2^{n-1} \\
     \vdots & \vdots& \ddots&\vdots \\
     1 & \lambda_n & \dots & \lambda_n^{n-1} 
 \end{bmatrix} = \prod \limits_{j<i} (\lambda_i - \lambda_j). $$ Then \begin{align*}
     \int_{\T^d}  \prod \limits_{j < i} \left| 1 - \frac{\lambda_{\sigma(i)}(UA)}{\lambda_{\sigma(j)}(UA)} \right|^{2} \ d\nu_{\T^d}(U) 
     &= \int_{\T^d}  \frac{\prod \limits_{j < i}|\lambda_{\sigma(j)} (UA) - \lambda_{\sigma(i)}(UA)|^2}{\prod \limits_{j < i}|\lambda_{\sigma(j)}(UA)|^2}  \\
     &= \int_{\T^d} \frac{|V(\lambda_1(UA),...,\lambda_n(UA))|^2}{\prod \limits_{j < i} |\lambda_j(UA)|^2} \ d\nu_{\T^d}(U)
 \end{align*} 

Leibniz formula gives 
$$|V(\lambda_1,...,\lambda_d)|^2 = \sum \limits_{\sigma, \tau \in \Sym(d)} \mathrm{sgn}(\sigma)\mathrm{sgn}(\tau) \lambda_1^{\sigma(1) - 1} \overline{\lambda}_1^{ \ \tau(1)-1} \dots \lambda_n^{\sigma(d)-1} \overline{\lambda}_d^{\ \tau(d)-1}.$$

Now $$\int_{0}^1 \lambda_k^{\sigma(k)-1} \overline{\lambda}_k^{\tau(k)-1} \ d\theta_k = |\lambda_k|^{\sigma(k) + \tau(k) - 2} \int_0^{1} e^{i\theta_k(\sigma(k) - \tau(k))} = \begin{cases}
    1 \ \ \text{if } \sigma(k) = \tau(k) \\
    0 \ \ \text{otherwise}.
\end{cases}$$

This gives that \begin{align*}
    \int_{\T^d}  \prod \limits_{j < i} \left| 1 - \frac{\lambda_{\sigma(i)}(UA)}{\lambda_{\sigma(j)}(UA)} \right|^{2} \ d\nu_{\T^d}(U)  = \sum \limits_{\sigma \in \Sym(d)} \frac{|\lambda_1|^{2\sigma(1) - 2} \dots |\lambda_d|^{2\sigma(d) - 2}}{\prod \limits_{j<i} |\lambda_j|^2}. 
\end{align*}

\begin{lemma}
    $$\sum \limits_{\sigma \in \Sym(d)}  \prod \limits_{j=1}^d |\lambda_j|^{2\sigma(j) - 2}=  \sum \limits_{\sigma \in \Sym(d)} \prod \limits_{j=1}^d |\lambda_j|^{2(d-\sigma(j))}$$
\end{lemma}

\begin{proof}
    Define a bijection $\tau:\Sym(d) \to \Sym(d)$ as $$\tau(\sigma)(j) = d + 1 - \sigma(j).$$ Then $$2\sigma(j) - 2 = 2(d-\tau(\sigma)(j)) = 2d - 2\sigma(j).$$
\end{proof}

\begin{proposition}[Theorem 22 in \cite{dedieu2003mike}] \label{almostalmost}
        Let $h:\FF_d \to \R$ be continuous. Then
        \begin{align*}
            \int_{f \in \FF_d} h(f) \ d\mu(f)
           &=  \int_{U \in U(d)} \sum_{\sigma \in \Sym(d)} h(\phi(v_{\sigma(1)},...,v_{\sigma(d)})) \frac{\prod \limits_{j=1}^d |\lambda_j|^{2(d-\sigma(j))}}{\sum \limits_{\pi \in \Sym(d)} \prod_{j=1}^d |\lambda_j|^{2(d-\pi(j))}  }\ d\nu(U).
        \end{align*}
\end{proposition}

By the discussion in section \ref{discussion}, this finishes the proof of Theorem \ref{DS measures}.

\subsection{Example in dimension 2}

To illustrate the calculation, we analyze the dimension 2 case. Note that $\FF_2 = \CP^1$. By Proposition \ref{almost}, for continuous $h: \CP^1 \to \R$, $$\int_{v \in \CP^1} h(v) \ d\mu(v) = \int_{U \in \SU(2)} h(v_1(UA)) \sum \limits_{\sigma \in \Sym(2)} \prod \limits_{j < i} \left| 1 - \frac{\lambda_{\sigma(i)(UA)}}{\lambda_{\sigma(j)(UA)}}\right| \ d\nu(U).$$

In Subsection \ref{computationally}, we have that for $\sigma = \text{id}$ \begin{align*}
    \int_{U \in \U(d)}\prod \limits_{j < i} \left| 1 - \frac{\lambda_{\sigma(i)}(UA)}{\lambda_{\sigma(j)}(UA)} \right|^2 \ d\nu(U) &= \int_{\theta_1,\theta_2 \in [0,1]} \left| 1 -  \frac{ e^{2\pi i\theta_1}\lambda_2}{e^{2 \pi i\theta_2}\lambda_1}\right|^{2} \ d\theta_1 d\theta_2 \\
    &= \int_0^{1} \left| 1 - e^{2 \pi i\theta} \frac{ \lambda_2}{\lambda_1}\right|^{2} \ d\theta \\
    &=\frac{\lambda_1^2 + \lambda_2^2}{\lambda_1^2}.
\end{align*}

Now, by doing the same computation with $\sigma = (1 \ 2)$, and inverting, gives Equation \ref{dim2weights}.

\section{Experimental Results: The Arnold Family}

\begin{figure}
    \centering
    \includegraphics[width=\textwidth]{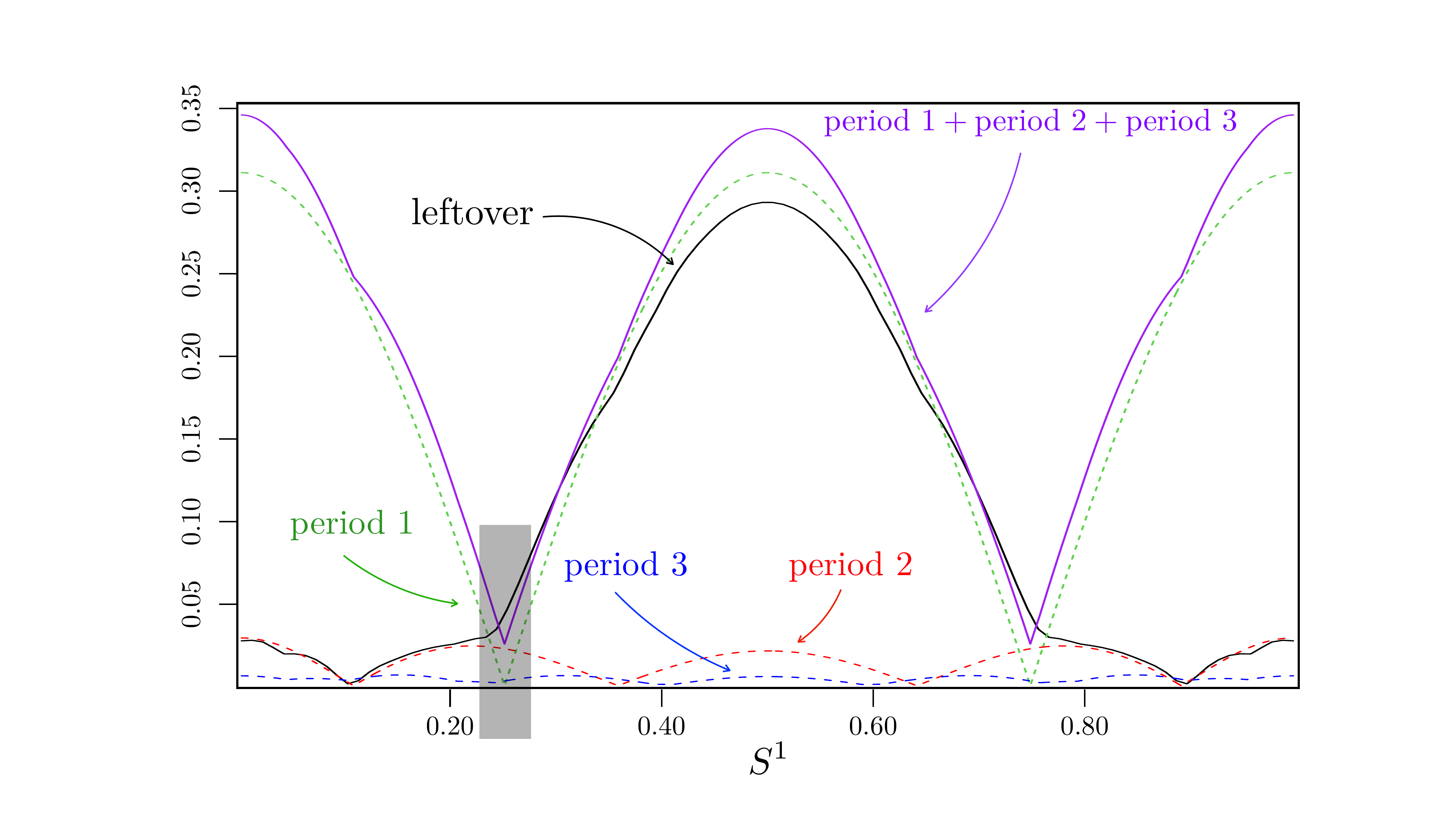}
    \caption{In black, which we call the leftover, we depict the density of $\cE(x) = 1 - \frac{d}{d\mu} \int_{r \in [0,1] \setminus \Q} \tau^+_{r,0.05} dr.$ 
    For there to be a Dedieu Shub  measure for the Arnold family, $\sum_{p/q \in [0,1] \cap \Q} \mu(I_{p,q}) \frac{d}{d\mu}\left(\int_{c \in I_{p/q}} \tau^+_{c,\epsilon} + \tau^i_{c,\epsilon} \ dc\right)(x) \geq  \cE(x) $ for every $x \in X$. However, in the grey shaded box, we see this is not possible.   
    }
    \label{fig:enter-label}
\end{figure}

One can ask whether or not Dedieu--Shub measures exist for larger groups, past the algebraic setting. The set of ergodic measures forms a Borel space with a measure $m$. Informally, the measure $m$ assigns a "weight" to each ergodic invariant measure. Typically, one takes this to be a probability, but we take it to assign each measure a mass of 1.  Observe that one very weak obstruction obstruction to the existence of Dedieu--Shub measures is if given a $g \in G$ we have that \begin{equation}
    \int_K \int_{ \cE_{gk}}  \mu \ \ dm(\mu) \ d \Haar(k) \not \geq \mu \ \ \label{obstruction}
\end{equation}
where $\cE_{f}$ is the space of all ergodic $f$--invariant measures. If \ref{obstruction} holds, for any $g \in G$, \emph{then there can be no Dedieu--Shub measure}. So to show that a Dedieu--Shub measure does not exist for a larger group, for example $G = \text{Diff}^\infty(S^1)$, our goal is to produce a single coset for which \ref{obstruction} holds.

Our candidate to demonstrate \textit{experimentally} that there is no Dedieu -- Shub measure for $\text{Diff}^\infty(S^1)$ is \textit{the Arnold family}, given by $f_{c, \epsilon} = x + c + \epsilon \sin (2 \pi x) \mod 1$ for $\epsilon \in [0,2\pi]$ and $c \in [0,1]$. 
The Arnold family fits within our framework. The compact subgroup acting on $S^1 = \R / \Z$ is the rotation group $K = \{R_c = x+c : c \in [0,1]\} = S^1$. The map $g(x) = x + \epsilon \sin 2 \pi x \mod 1$ converges to a fixed point, and only by composing it with the family of rotation $R_c(x) = x+c$ produces a comparatively rich dynamics. Much is known about the Arnold family, and we summarize the relevant facts in the following theorem.

\begin{theorem}[\cite{arnold1965small} ]
    Fix $\epsilon > 0$. For every rational number $p/q$, there exists an interval $I_{p,q} \in [0,1]$ of positive measure such that for every $c \in I$, $f_{c, \epsilon}$ has rotation number $p/q$.
  
Furthermore, if $f_{c,\epsilon}$ has rational rotation number $p/q$, then $f_{c,\epsilon}$ has a unique repelling limit cycle and a unique attracting limit cycle, both of period $q$. If $f_{c,\epsilon}$ has irrational rotation number, it preserves a unique absolutely continuous invariant probability measure and is smoothly conjugate to an irrational rotation.
\end{theorem}

Because of the existence of an attractor and repeller, if $f_{c,\varepsilon}$ has rational rotation number, we call it \textit{hyperbolic}, and if it is irrational, then we call it \textit{elliptic}. In the hyperbolic setting, there are two ergodic measures supported on the repelling and attracting periodic orbits -- which we denote by $\tau_{c,\varepsilon}^-$ and $\tau_{c,\varepsilon}^+$ respectively -- and in the elliptic setting, there is simply one ergodic measure -- which we denote by $\tau_{c,\varepsilon}^+$ -- because the dynamics are uniquely ergodic as it is conjugate to an irrational rotation.
We have the following observation after numerical computations.

\begin{claim}[Experimental] \label{experiment}
When $\varepsilon \in 0.05$, $$\sum_{p/q \in \Q} \mu(I_{p,q}) \left(\int_{c \in I_{p,q}} \tau_{c, \epsilon}^+ + \tau_{c, \epsilon}^- \right) + \sum_{r \in \R \setminus \Q} \tau_{r,\epsilon}^+ \not \geq \mu$$ where $\mu$ is Lebesgue on $[0,1]$.
\end{claim}

% \begin{conjecture}
% When $\varepsilon \in (0,1/2\pi]$, $$\sum_{p/q \in \Q} \mu(I_{p,q}) \left(\int_{I_{p,q}} \tau^+ + \tau^- \right) + \sum_{r \in \R \setminus \Q} \tau^+ < \mu$$ where $\mu$ is Lebesgue on $[0,1]$.
% \end{conjecture}

This experimental result should be placed in context with the work of de La Llave, Sìmo, and Shub \cite{dLL2008entropy}, who studied the expanding family $$g_{k,\varepsilon,c} = kx + c + \varepsilon \sin 2 \pi x$$ when $2 \leq k \in \N.$ As this family is expanding and smooth, the forward physical measures is smooth and unique.  They explicitly computed the density of the average of all of these smooth measures, and concluded it was not Lebesgue. Of course, this does not forbid the existence of Dedieu--Shub measures in the expanding setting where there are infinitely many invariant measures (that are convex combinations of infinitely many periodic orbits) but is an indication that they may not exist.

% \begin{figure}
%     \centering
%         % \begin{minipage}{0.48\textwidth}
%         % \centering
%         \includegraphics[width=0.8\textwidth]{12123.pdf} % second figure itself
%     % \end{minipage}
%     % \hfill
%     % \begin{minipage}{0.48\textwidth}
%     %     \centering
%         \includegraphics[width=0.8\textwidth]{epsilon15.pdf} % first figure itself
%     % \end{minipage}
%     \label{fig:rotation}
%         \caption{The $x$--axis parametrizes $S^1$ and the $y$--axis in two different $\varepsilon$ regimes -- when $\varepsilon = 0.05 $ and $0.1$. As the nonlinearity parameter $\varepsilon$ increases, we see the curves become less tame. A point $(x,c)$ in this graph corresponds to a $p/q$ periodic point in the corresponding color.  If the derivative at that periodic point is negative, it is a repeller, and if the derivative is positive, it is an attractor. }
% \end{figure}

\subsection{Description of experiment.}
Consider the phase space $[0,1] \times [0,1] = S^1 \times  \{ c - \text{parameters } \}$. Fix rational $p/q$ in lowest terms. If a periodic point is associated to a dynamical system with rotation number $p/q$, we call that point $p/q$ periodic. For fixed $\varepsilon$, a $p/q$--periodic point $x$, satisfies $f_{c,\varepsilon}^q(x) = x + p$ where the power denotes composition. This can be numerically computed using Newton's method. Fixing $\epsilon$, if an $x$ satisfies this equation, then $x$ is $p/q$ periodic for the parameter $c$. One can compute the graph $$g_{p/q} := \{(x,c) \in S^1 \times [0,1]: x \text{ is a } p/q \text{ periodic point} for f_{c,\epsilon}\}$$ and we call this graph a  rotation curve. Reiterating, the range of a given $g_{p/q}$ gives all possible $c$ such that $f_{c,\varepsilon}$ is a $p/q$ periodic point. It is clear that $$\frac{d}{d\mu} \int_{c \in I_{p/q}} \tau_{c,\epsilon}^+ + \tau_{c,\epsilon}^+ = \frac{d}{d\mu} g_{p/q}.$$

In particular, the densities of the  periodic points of some fixed rotation number is precisely the absolute value of the derivative of these rotation curves. 
% \begin{figure}
%     \centering
%     \includegraphics[width = 0.8\textwidth]{Screen Shot 2024-01-29 at 12.39.30 PM.png}
%     \caption{Probability density function of all period 2 points when $\varepsilon = 0.05$, normalized by the mass of the period 2 region. This is equal to the absolute value of the derivative of the appropriate rotation curve above in Figure 4.}
%     \label{fig:density2}
% \end{figure}

Now let us restrict our attention to when $\varepsilon = 0.05$. If $c \in [0,0.05] \cup [0.95,1]$, then $f_{c,\varepsilon}$ consists of a single attracting fixed point and a single repelling fixed point. If $c \in \sim [0.4961, 0.5039]$ (of course it is a bit sharper), then $f_{c,\varepsilon}$ has rotation number 1/2. The measure of the period-2 region is clearly less than 0.08. The period 3 region, is even smaller, with measure approximately $0.00214$. As the period increases, the measure of the corresponding $c$ parameters decrease exponentially. The leftover, the parameters producing dynamics conjugated to an irrational rotation, constitutes about $88\%$ of the measure of parameters, and obviously, the preserved measure is unique. To generate the ``elliptic distributions", we picked one million $f_{c,0.05}$ randomly chosen, that was not hyperbolic. For each one, we computed one million forward iterations, and computed a histogram with 100 bins evenly space in [0,1]. Each histogram is a numerical approximation of the smooth absolutely continuous measure.  

% In, we show in orange what the average of all these elliptic densities looks like. In pink, we subtract it from one. If a Dedieu--Shub measure exists, we should be able to weight the attracting and repelling regions in each map, so that the remaining mass fills the pink region. 

% \begin{figure}
%     \centering
%     \includegraphics[width = 0.9\textwidth]{elliptic.pdf}
%     \caption{Caption}
%     \label{fig:leftoverelliptic}
% \end{figure}

% \begin{figure}
%     \centering
%     \includegraphics[width = \textwidth]{sliders.png}
%     \caption{The region boxed in dotted brown can never match the what is leftover by the elliptic distribution.}
%     \label{fi4abel}
% \end{figure}

For the hyperbolic densities, we numerically computed them from the rotation curves. We are primarily concerned with the period 1, period 2 points, and period 3 as the densities of higher period points are negligble. 

As is shown in Figure \ref{fig:enter-label} we plot a numerical approximation for the quantities in \ref{experiment}. We see that in the shaded region the mass cannot be equal to the leftover, which is equal to $1 - $ the integrated elliptic densities. Hence \ref{obstruction} holds numerically and so it seems there is no Dedieu--Shub measure.

\printbibliography

\end{document}